\documentclass[12pt,reqno]{amsart}
  \usepackage[all,2cell,arrow]{xy}
  \usepackage{amsfonts}
  \usepackage{amsthm}
  \usepackage{amsmath}
  \usepackage{amssymb}
\usepackage{mathtools}
\usepackage{tensor}
 \numberwithin{equation}{section}
 \usepackage{xcolor}
\usepackage{graphicx}
\usepackage[hidelinks]{hyperref}
\usepackage{enumerate,xspace}

\UseAllTwocells
\CompileMatrices

 \def\black{\color{black}}

\renewcommand{\epsilon}{\varepsilon}
\renewcommand{\phi}{\varphi}

\newcommand{\ca}{\ensuremath{\mathcal A}\xspace}
\newcommand{\cb}{\ensuremath{\mathcal B}\xspace}
\newcommand{\cc}{\ensuremath{\mathcal C}\xspace}
\newcommand{\cd}{\ensuremath{\mathcal D}\xspace}
\newcommand{\ce}{\ensuremath{\mathcal E}\xspace}

\newcommand{\cg}{\ensuremath{\mathcal G}\xspace}
\newcommand{\ch}{\ensuremath{\mathcal H}\xspace}
\newcommand{\ci}{\ensuremath{\mathcal I}\xspace}
\newcommand{\cj}{\ensuremath{\mathcal J}\xspace}
\newcommand{\ck}{\ensuremath{\mathcal K}\xspace}

\newcommand{\cv}{\ensuremath{\mathcal V}\xspace}

\newcommand{\bbz}{\ensuremath{\mathbb Z}\xspace}

\newcommand{\atwo}{{\mathbf 2}}

\newcommand{\dsdr}{shrinkable}
\newcommand{\hh}{\textnormal{h}}
\newcommand{\id}{\textnormal{id}}
\newcommand{\cat}{\textnormal{-}\ensuremath{\mathbf{Cat}}\xspace}

\newcommand{\Ab}{\ensuremath{\mathbf{Ab}}\xspace}

\newcommand{\ER}{\ensuremath{\mathbf{ER}}\xspace}

\newcommand{\Cat}{\ensuremath{\mathbf{Cat}}\xspace}
\newcommand{\Set}{\ensuremath{\mathbf{Set}}\xspace}
\newcommand{\SSet}{\ensuremath{\mathbf{SSet}}\xspace}
\newcommand{\QCat}{\ensuremath{\mathbf{QCat}}\xspace}

\newcommand{\Cart}{\ensuremath{\mathbf{Cart}}\xspace}
\newcommand{\Gray}{\ensuremath{\mathbf{Gray}}\xspace}
\newcommand{\cgtop}{\ensuremath{\mathbf{CGTop}}\xspace}

\newcommand{\op}{\ensuremath{^{\textnormal{op}}}}

\newcommand{\fib}{\ensuremath{_{\textnormal{fib}}}}

\newcommand{\ev}{\textnormal{ev}}

\newcommand{\Eq}{\ensuremath{\mathbf{Eq}}\xspace}

\DeclareMathOperator{\cod}{cod}
\DeclareMathOperator{\dom}{dom}

\DeclareMathOperator{\Inj}{Inj}
\DeclareMathOperator{\Ch}{Ch}

\DeclareMathOperator{\Ps}{Ps}

\DeclareMathOperator{\colim}{colim}

\newcommand{\TAlg}{{\textnormal{T-Alg}}}
\newcommand{\SAlg}{{\textnormal{S-Alg}}}
\newcommand{\cone}[1]{#1\operatorname{{}-cone}}
\newcommand{\en}[1]{\mathbb{#1}}

\newcommand{\two}{\ensuremath{\mathbf{2}\xspace}}

\newcommand{\ox}{\otimes}
\newcommand{\x}{\times}

\newcommand{\<}{\langle}
\renewcommand{\>}{\rangle} 

\def\1c#1{\stackrel{#1}{\to}}

  \newtheorem{proposition}{Proposition}[section]
  \newtheorem{lemma}[proposition]{Lemma}
  
  \newtheorem{corollary}[proposition]{Corollary}
  \newtheorem{theorem}[proposition]{Theorem}

  \theoremstyle{definition}
  \newtheorem{definition}[proposition]{Definition}
  \newtheorem{notation}[proposition]{Notation}
  
  \newtheorem{example}[proposition]{Example}
  \newtheorem{examples}[proposition]{Examples}

  \theoremstyle{remark}
  \newtheorem{remark}[proposition]{Remark}

  \newcounter{c}
  \renewcommand{\[}{\setcounter{c}{1}$$}
  \newcommand{\etyk}[1]{\vspace{-7.4mm}$$\begin{equation}\Label{#1}
  \addtocounter{c}{1}}
  \renewcommand{\]}{\ifnum \value{c}=1 $$\else \end{equation}\fi}
\begin{document}
\title[Adjoint functor theorems]{Adjoint functor theorems for
  homotopically enriched categories}
\author{John Bourke}
\address{Department of Mathematics and Statistics, Masaryk University, Kotl\'a\v rsk\'a 2, Brno 61137, Czech Republic}
\email{bourkej@math.muni.cz}

\author{Stephen Lack}
\address{Department of Mathematical and Physical Sciences, Macquarie
  University NSW 2109, Australia}
\email{steve.lack@mq.edu.au}
\thanks{The first and third named authors acknowledge the support of the Grant Agency of the Czech Republic under the grant 19-00902S. The second-named author acknowledges with gratitude the hospitality of Masaryk University, and the support of an
  Australian Research Council Discovery Project DP190102432.}

\author{Luk\'{a}\v{s} Vok\v{r}\'{i}nek}
\address{Department of Mathematics and Statistics, Masaryk University,
  Kotl\'a\v rsk\'a 2, Brno 61137, Czech Republic}
\email{koren@math.muni.cz}

\date{\today}

\begin{abstract}
We prove an adjoint functor theorem in the setting of categories
enriched in a monoidal model category $\cv$ admitting certain limits.
When $\cv$ is equipped with the trivial model structure this
recaptures the enriched version of Freyd's adjoint functor theorem.
For non-trivial model structures, we obtain new adjoint functor theorems of a homotopical flavour --- in particular, when 
$\cv$ is the category of simplicial sets we obtain a homotopical adjoint functor theorem appropriate to the $\infty$-cosmoi of Riehl and Verity.  We also investigate accessibility in the enriched setting, in particular obtaining homotopical cocompleteness results for accessible $\infty$-cosmoi.
\end{abstract}

\maketitle 

\section{Introduction}

The general adjoint functor theorem of Freyd describes conditions under which a functor $U\colon \cb \to \ca$ has a left adjoint: namely,
\begin{itemize}
\item $U$ satisfies the solution set condition and 
\item $\cb$ is complete and $U\colon \cb \to \ca$ preserves limits.
\end{itemize}
The first applications of this result that one typically learns about
are the construction of left adjoints to forgetful functors between
algebraic categories, and the cocompleteness of algebraic categories.
In the present paper we shall describe a strict generalisation of Freyd's result, applicable to categories of a homotopical nature, and shall demonstrate similar applications to homotopical algebra.  

The starting point of our generalisation is to pass from ordinary $\Set$-enriched categories to categories enriched in a \emph{monoidal model category} $\cv$.   For instance, we could take $\cb$ to be
\begin{itemize}
\item the $2$-category of monoidal categories and strong monoidal functors (here $\cv = \Cat$);
\item the $\infty$-cosmos of quasicategories admitting a class of limits and morphisms preserving those limits (here $\cv= \SSet$);
\item the category of fibrant objects in a model category (here $\cv = \Set$ but with what we call the split model structure)
\end{itemize}
and $U\colon \cb \to \ca$ the appropriate forgetful functor to the base.  

Now in none of these cases is the $\cv$-category $\cb$ complete in the
sense of enriched category theory, but each does admit certain limits
captured by the model structure in which we are enriching --- in
particular, \emph{cofibrantly-weighted limits}.

If we keep the usual solution set condition, which holds in many
examples including those described above, the question turns to what
kind of left adjoint we can hope for?
Our answer here will fit into the framework of \emph{enriched
  weakness} developed by the second-named author and Rosick\'y in
\cite{LackRosicky2012}.
Given a class $\ce$ of morphisms in $\cv$, one says that $U\colon\cb
\to \ca$ admits an $\ce$-weak left adjoint if for each object $A\in \ca$
there exists an object $A^{'} \in \cb$ and a morphism
$\eta_A\colon A\to UA^{'}$ for which the induced map
\begin{equation}\label{eq:can}
\xymatrix @C3pc {
\cb(A^{'},B) \ar[r]^-U & \ca(UA^{'},UB) \ar[r]^-{\ca(\eta_A,UB)}  & \ca(A,UB)
}
\end{equation}
belongs to $\ce$ for all $B\in\cb$.
Here one recovers classical adjointness on taking $\ce$ to be the
isomorphisms.  In this paper, we take $\ce$ to be the class of morphisms in the
model category $\cv$ with the {\em dual strong deformation retract
  property} (see Section~\ref{sect:dsdr} below). 
We follow \cite{Dold-shrinkable} in calling such morphisms
{\em shrinkable},\footnote{And we thank Karol Szumi\l o for suggesting
  the name.} and we introduce the name {\em
  left shrink-adjoint} for $\ce$-weak left adjoint in this case where
$\ce$ consists of the shrinkable morphisms.

Our first main result, Theorem~\ref{thm:adjunction}, is our
generalisation of Freyd's adjoint functor theorem, and gives a
sufficient condition for a \cv-functor to have such a left
 shrink-adjoint. 
It {\em is} a generalisation, since when $\cv=\Set$ with the
trivial model structure, all colimits are cofibrantly-weighted, the
shrinkable morphisms are the isomorphisms, and left
shrink-adjoints are just ordinary left adjoints.  

Our second main result, Theorem~\ref{thm:colimits}, applies our adjoint functor theorem
to the construction of $\ce$-\emph{weak colimits}, in the sense
of \cite{LackRosicky2012}.  Namely, we prove that if $\cc$
is an \emph{accessible} $\cv$-category admitting certain limits, then
$\cc$ admits $\ce$-weak colimits. Once again, in our current
setting where $\ce$ consists of the shrinkable morphisms, we introduce
the term {\em shrink-colimit} for $\ce$-weak colimit. 

The remainder of the paper is devoted to giving applications of our two main
results in various settings.  In particular, we interpret them for $2$-categories and $\infty$-cosmoi,
where we obtain new results concerning adjoints of a homotopical flavour and homotopy colimits.  We also provide numerous concrete examples to which these results apply.

Let us now give a more detailed overview of the paper.
In Section~\ref{sect:setting} we introduce our basic setting, including
four running examples to which we return throughout the paper.  These four examples capture \emph{classical category theory},
\emph{weak category theory}, \emph{2-category theory} and \emph{$\infty$-cosmoi}.
 Furthermore, we assign to each $\cv$-category $\cc$ its ``homotopy
$\ER$-category'' $\hh_{*}\cc$: a certain category enriched in
equivalence relations. Since this is a simple kind of 2-category, it admits notions of equivalence, bi-initial objects, and so on --- accordingly, we can transport these notions to define equivalences and bi-initial objects in the $\cv$-category $\cc$.

The next five sections walk through generalisations of the various
steps in the proof of the adjoint functor theorem given in \cite{CWM}, which we summarise below.
\begin{itemize}
\item The technical core of the general adjoint functor theorem is the result that a complete category $\cc$ with a weakly initial set of objects has an initial object.  \emph{In Section~\ref{sect:initial} we generalise this result, establishing conditions under which $\cc$ has a bi-initial object.  }
\item In the classical setting the next step is to apply the preceding
  result to obtain an initial object in each comma category $A/U$.
  \emph{Accordingly Section~\ref{sect:comma} is devoted to the subtle world of enriched comma categories.}
\item The next step, trivial in the classical case, is the observation
  that if $\eta\colon A \to UA^{'}$ is initial in $A/U$ then
  \eqref{eq:can} is invertible.
  Another way of saying this is that \eqref{eq:can} is \emph{terminal}
  in the slice $\Set/\cb(A,UB)$.
  \emph{ In Section~\ref{sect:dsdr} we consider the class $\ce$ of \dsdr\ morphisms, which turn out to be bi-terminal objects in enriched slice categories.}
\item Putting all this together in Section~\ref{sect:I=1}, we prove our generalised adjoint functor theorem in the case where the unit $I\in\cv$ is terminal. In Section~\ref{sect:general} 
we adapt this to cover the general case by passing from $\cv$-enrichment to $\cv/I$-enrichment.
\end{itemize}
In Section~\ref{sect:cocompleteness} we  consider accessible
$\cv$-categories and prove that each accessible $\cv$-category with
powers and enough cofibrantly-weighted limits admits
shrink-colimits. 

In Section~\ref{sect:applications} we interpret our two main theorems in each of our four running examples.  In the first two examples we recover, and indeed extend, the classical theory.  Our focus, however, is primarily on $2$-categories and $\infty$-cosmoi, where we describe many examples to which our two main results apply.  In particular, we show that many natural examples of $\infty$-cosmoi are accessible, which allows us to conclude that they admit \emph{flexibly-weighted homotopy colimits} in the sense of \cite{Riehl2018On}.

\section{The setting}\label{sect:setting}

We start with a monoidal model category \cv, as in \cite{Hovey1999Model}
for example, in which the unit object $I$ is cofibrant.
We may sometimes allow ourselves to write as if the monoidal structure were strict. 

We shall now give four examples, to which we shall return throughout
the paper. Several further examples may be found in
Section~\ref{sect:dsdr} below.

\begin{example}\label{ex:trivial}
Any complete and cocomplete  symmetric monoidal closed category
$(\cv,\otimes,I)$ may be equipped with the \emph{trivial model
  structure}, in which the weak equivalences are the isomorphisms, and
all maps are cofibrations and fibrations.
The case of classical category theory corresponds to the cartesian closed category $(\Set,\times, 1)$.
\end{example}

\begin{example}\label{ex:split}
 For our second example, concerning weak category theory, we take our motivation from the case of $\Set$, in which case the relevant model structure has cofibrations the injections, fibrations the surjections, and all maps as weak equivalences.  

For general $\cv$ we refer to the corresponding model structure as the
\emph{split model structure}, which we now define.  By Proposition~2.6
of \cite{Rosicky2007Factorization}, in any category with binary
coproducts there is a weak factorisation whose left class consists of
the retracts of coproduct injections $ inj\colon X \to X + Y$, while the right
class consists of the split epimorphisms.
Here the canonical factorisation of an arrow $f\colon X \to Y$ is given by the cograph factorisation 
\begin{equation*}
\xymatrix{X \ar[r]^-{inj} & X+Y \ar[r]^-{\langle f,1 \rangle} & Y.}
\end{equation*}
As with any weak factorisation system (on a category with, say, finite
limits and colimits), we can extend it to a model structure whose
cofibrations and fibrations are the two given classes, and for which
all maps are weak equivalences.
Doing this in the present case
produces the split model structure.
It is a
routine exercise to show that if $(\cv,\otimes,I)$ is a complete and
cocomplete symmetric monoidal closed category, then the split model structure on $\cv$ is monoidal.   
\end{example}

\begin{example}\label{ex:Cat}
  The cartesian closed category $(\Cat,\times,1)$ becomes a monoidal
  model category when it is equipped with the {\em canonical model
  structure}: here the weak equivalences are the equivalences of
  categories, the cofibrations are the functors which are injective on
  objects, and the fibrations are the isofibrations. All objects are
  cofibrant and fibrant. 
  The trivial fibrations are the equivalences which are surjective on objects; these are
  often known as surjective equivalences or retract equivalences. 
\end{example}

\begin{example}\label{ex:Joyal}
  Our final main example is the cartesian closed category
  $(\SSet,\times,1)$ of simplicial sets equipped with the Joyal model
  structure.
  (Using the Kan model structure would give another example, but it
  turns out that this leads to a strictly weaker adjoint functor
  theorem.)
\end{example}

We follow the standard notation in enriched category theory of writing $\ca_0$ for the underlying ordinary category of a $\cv$-category $\ca$. This has the same objects as $\ca$, and the hom-set $\ca_0(A,B)$ is given by the set of morphisms $I\to \ca(A,B)$ in $\cv$. Similarly, we write $F_0\colon\ca_0\to\cb_0$ for the underlying ordinary functor of a $\cv$-functor $F\colon\ca\to\cb$. 

We shall begin by showing how to associate to any \cv-category a
2-category whose hom-categories are equivalence relations (that
is, groupoids which are also preorders).

\subsection{Intervals}

An {\em interval} in \cv will be an object $J$ equipped with a
factorisation
\[ \xymatrix{
    I+I \ar[r]^-{(d~c)} & J \ar[r]^e & I } \]
of the codiagonal for which $(d~c)$ is a cofibration and $e$ a weak
equivalence. We shall then say that ``$(J,d,c,e)$ is an
interval''. Observe that if $(d~c)$ is a cofibration, then since $I$
is cofibrant, $d$ and $c$ are also cofibrations, and the following are
equivalent:
\begin{itemize}
\item $e$ is a weak equivalence
\item $d$ is a trivial cofibration
  \item $c$ is a trivial cofibration. 
\end{itemize}

We may always obtain an interval by factorising the codiagonal $I+I\to
I$ as a cofibration followed by a trivial fibration; an interval of
this type will be called a {\em standard interval}.
$(J,d,c,e)$ is an interval if and only if $(J,c,d,e)$ is one.
There are also two constructions on intervals that will be needed.

\begin{proposition}\label{prop:interval-for-transitivity}
  If $(J_1,d_1,c_1,e_1)$ and $(J_2,d_2,c_2,e_2)$ are intervals, so is $(J, d'_2d_1, c'_1c_2,e)$ where $J$ is constructed as the
  pushout
  \[ \xymatrix @C1pc @R1pc {
      & J_1 \ar[dr]^{d'_2} \ar@/^1pc/[drr]^{e_1} \\
      I \ar[ur]^{c_1} \ar[dr]_{d_2} && J \ar[r]^e & I  . \\
      & J_2 \ar[ur]_{c'_1} \ar@/_1pc/[urr]_{e_2} } \]
\end{proposition}

\proof
Since $d'_2$ and $c'_1$ are pushouts of trivial cofibrations they are
trivial cofibrations. Thus both $d'_2d_1$ and $c'_1c_2$ are trivial
cofibrations. On the other hand we can construct $J$ as the pushout
\[ \xymatrix{
    & I+I+I+I \ar[rr]^-{(d_1~c_1)+(d_2~c_2)} \ar[d]_{1+\nabla+1} && J_1+J_2
    \ar[d]^{d'_2+c'_1} \\
    I+I \ar[r]_-{i_{1,3}} & I+I+I \ar[rr] && J } \]
and so the lower horizontal is also a cofibration.
\endproof

\begin{proposition}
  \label{prop:interval-for-tensor}
  If $(J_1,d_1,c_1)$ and $(J_2,d_2,c_2)$ are intervals then $J_1\ox
  J_2$ becomes an interval when equipped with the maps
  \[ \xymatrix @C2pc @R1pc {
      &I \ar[r]^-{\cong} & I\ox I \ar[r]^-{d_1\ox d_2} & J_1\ox J_2  \\
      & I \ar[r]^-{\cong} & I\ox I \ar[r]^-{c_1\ox c_2} & J_1\ox J_2  \\
     & J_1\ox J_2 \ar[r]^-{e_1\ox e_2} & I\ox I \ar[r]^-{\cong} & I. 
} \]
\end{proposition}

\subsection{\ER-categories}

Let \ER be the full subcategory of \Cat consisting of the {\em
  equivalence relations}: that is, the groupoids which are also
preorders.
This is closed in $\Cat$ under both products and internal homs; 
it follows that \ER becomes a cartesian closed category in its own right, and that we can consider
categories enriched over \ER.

An \ER-category is just a 2-category for which each hom-category is an
equivalence relation. Equivalently, it is an ordinary category in
which each hom-set is equipped with an equivalence relation, and this
relation is respected by composition on either side (whiskering, if
you will). 

Since an \ER-category is a special sort of 2-category, we can consider
standard 2-categorical notions. Thus morphisms $f,g\colon A\to B$ are
isomorphic in the 2-category just when they are related under the given
equivalence relation. We may as well therefore use $\cong$ to denote
this relation. An {\em
  equivalence} in an \ER-category consists of morphisms $f\colon A\to
B$ and $g\colon B\to A$ for which both $gf\cong 1_A$ and $fg\cong 1_B$.

\begin{definition}
An object $T$ of an \ER-category $\cc$ is {\em bi-terminal} if for
each object $C$, the induced map $\cc(C,T)\to 1$ is an equivalence.
\end{definition}
In elementary terms, this means that there exists a morphism $C \to T$ and it is unique up to isomorphism.  Dually there are {\em bi-initial} objects.

\subsection{The homotopy \ER-category of a \cv-category}

Any monoidal category has a monoidal functor to \Set given by homming
out of the unit object; in our case, this has the form
 $\cv(I,-)\colon \cv\to \Set$. The monoidal structure is defined by
\begin{equation}\label{eq:monoidal}
 \xymatrix @R0pc {
    \cv(I,X)\x\cv(I,Y) \ar[r] & \cv(I,X\ox Y)   \\
    (x,y) \ar@{|->}[r] & x\ox y
    } 
  \end{equation}
with unit defined by the map $1\to \cv(I,I)$ picking out the identity.
    
 Given $X\in\cv$ and $x,y\colon I\to X$, write $x\sim y$ if there is an
interval $(J,d,c,e)$ and a morphism $h\colon J\to X$ with $hd=x$ and
$hc=y$. This clearly defines a symmetric reflexive relation on
$\cv(I,X)$; and it is transitive by Proposition~\ref{prop:interval-for-transitivity}.
Given $f\colon X \to Y$, the function $\cv(I,f)$ is a morphism of equivalence relations, since if $x=hd\sim hc=y$ then also $fx=fhd\sim
fhc=fy$ for any $f$. Thus we have a functor $\cv\to \ER$.  
Thanks to Proposition~\ref{prop:interval-for-tensor}, the functions ~\eqref{eq:monoidal}
lift to morphisms of equivalence relations.  Since the unit map $1\to \cv(I,I)$ is trivially
a morphism of equivalence relations, the monoidal functor
$\cv(I,-)\colon \cv\to\Set$ lifts to a monoidal functor
$\hh\colon \cv\to\ER$. 

This $\hh$ induces a 2-functor $\hh_*\colon \cv\cat\to\ER\cat$.
Explicitly, for any \cv-category \cc, the induced \ER-category
$\hh_*(\cc)$ is given by the underlying ordinary category of \cc,
equipped with the equivalence relation $\cong$ on each hom-set
$\cc_0(A,B)$, where $f\cong g$ just when, seen as maps $I \to
\cc(A,B)$, they are $\sim$-related; in other words, if there is a factorisation
\[ \xymatrix{
    I+I \ar[r]^-{(f~g)} \ar[d]_-{(d~c)} & \cc(A,B) \\ J \ar@{->}[ur]_{h}
  } \]
for some interval $(J,d,c,e)$.   We call this relation $\cong$
the {\em\cv-homotopy relation}. 

\begin{remark}
For $f$ and $g$ as above, let us refer to a given extension $h$ as
above as a {\em $(\cv,J)$-homotopy}, and denote it as $h\colon f \cong_{J} g$.  We say that such an $h\colon J \to \cc(A,B)$ is \emph{trivial} if it factorises through $e\colon J \to I$, in which case of course $f=g$.

Given $a\colon X \to A$ and $b\colon B \to Y$ we naturally obtain
$(\cv,J)$-homotopies $h \circ a\colon f \circ a \cong_{J} g \circ a$ and $b \circ h\colon b \circ f \cong_{J} b \circ g$.  It is only when considering transitivity of the relation $\cong$ that we need to change the interval $J$.
\end{remark}

\begin{remark}
  A priori, we now have two ``homotopy relations'' for morphisms
  $x,y\colon I\to X$ in \cv: on the one hand $x\sim y$ and on the
  other $x\cong y$. But these are clearly the same.
  
More generally, morphisms $f,g\colon X\to Y$ in \cv are $\cv$-homotopic if any (and thus all) of the following
  diagrams has a filler:
  \[ \xymatrix{
      X+X \ar[r]^-{(f~g)} \ar[d]_{(d\cdot X~c\cdot X)} & Y & I+I \ar[r]^{(f~g)}
      \ar[d]_{(d~c)} & [X,Y] & X \ar@{.>}[r] \ar[dr]_{\binom{f}{g}}& Y^J \ar[d]^{\binom{Y^d}{Y^c}} \\
      J\cdot X \ar@{.>}[ur] && J \ar@{.>}[ur] &&& Y\x Y. } \]
 In the case that $X$ is cofibrant then $e\cdot X\colon J \cdot X \to X$ is a weak equivalence so that $J\cdot X$ is a \emph{cylinder object} for $X$; in this case the left diagram above shows $\cv$-homotopy implies \emph{left homotopy} in the model categorical sense of \cite{Quillen1967Homotopical}.  Similarly if $Y$ is fibrant the third diagram shows that $\cv$-homotopy implies \emph{right homotopy}, but in general there need be no relation between $\cv$-homotopy and left or right homotopy.
\end{remark}
Since a \cv-functor $F\colon \ca\to\cb$ induces an \ER-functor
$\hh_*(F)\colon \hh_*(\ca)\to \hh_*(\cb)$, if $f\cong g$ in \ca then
$Ff\cong Fg$ in \cb. This is one useful feature of $\cv$-homotopy not true of the usual left and right homotopy relations in the model category $\cv$.

\begin{definition}
An object of a \cv-category \ca is
{\em bi-terminal} or {\em bi-initial} if it is so in $\hh_*(\ca)$. 
\end{definition}
In elementary terms, $T$ is
bi-terminal in \ca if for each $A\in\ca$ there is a morphism $A\to T$,
which is unique up to $\cv$-homotopy. 

If \ca has an actual terminal
object $1$, then $T$ is bi-terminal if and only if the unique map
$! \colon\black T\to 1$ is an equivalence. Explicitly, this means that there is a
morphism $t\colon 1\to T$ with $t\circ !\cong 1_T$. 
The dual remarks apply to bi-initial objects.

\begin{proposition}\label{prop:adj-pres-bi-term}
  Right adjoint \cv-functors preserve bi-terminal objects. 
\end{proposition}

\proof
Let $U\colon\cb\to\ca$ be a \cv-functor with $F\dashv U$ a left
adjoint, and let $T\in\cb$ be bi-terminal.  For any $A\in\ca$, there is a morphism $FA\to T$ in \cb and so a
morphism $A\to UT$ in \ca.  If $f,g\colon A\to UT$ are two such morphisms in \ca, then their
adjoint transposes $f',g'\colon FA\to T$ are $\cv$-homotopic in \cb, hence
$f=Uf'\circ\eta_A\cong Ug'\circ \eta_A = g$ in \ca.
\endproof

\section{Shrinkable morphisms and shrink-adjoints}\label{sect:dsdr}

\begin{definition}
A morphism $f \colon A \to B$ in a $\cv$-category $\cc$ is said to be
 {\em \dsdr} if there exists a morphism $s\colon B\to A$ with
$f\circ s=1_B$, an interval $J$, and a $(\cv,J)$-homotopy $h\colon sf \cong_{J} 1_{A}$ such that $fh\colon f=fsf \cong_{J} f$ is trivial.  
\end{definition}

In elementary terms, such an $h$ is a morphism as below
\begin{equation}\label{eq:dsdr}
\xymatrix{
I+I \ar[d]_{(d~c)} \ar[rr]^{(1_A~s\circ f)} &&  \cc(A,A) 
\ar[d]^{\cc(A,f)} \\
J \ar[r]_{e} \ar[urr]_{h} & I \ar[r]_-f & \cc(A,B) } 
\end{equation}
rendering the diagram commutative.

Let us record a few easy properties.

\begin{proposition}
Shrinkable morphisms are closed under composition, contain the isomorphisms, and are preserved by any $\cv$-functor.
\end{proposition}
\begin{proof}
  Closure under composition follows from the fact --- see
  Proposition~\ref{prop:interval-for-transitivity} --- that intervals
  can be composed.  Verification of the remaining facts is routine.
\end{proof}

\begin{proposition}\label{prop:tf-implies-e}
  Let $\cc$  be a model \cv-category. 
 Any trivial fibration $f\colon A\to B$ in  $\cc$ with cofibrant domain
 and codomain is \dsdr.
\end{proposition}

\proof
Since $B$ is cofibrant, $f$ has a section $s$.  Since $A$ is cofibrant, the induced $\cc(A,f) \colon \cc(A,A) \to \cc(A,B)$ is a trivial fibration and so we obtain a diagonal filler $h \colon J\to  \cc(A,A)$ as in \eqref{eq:dsdr}.
\endproof

\begin{proposition}\label{prop:e-implies-we}
  Let $\cc$  be a model \cv-category.
  If $f\colon A\to B$ is \dsdr\ in  $\cc$  and either $d\cdot A\colon A\to J\cdot A$
  or $d\pitchfork A\colon J\pitchfork A\to A$ is a weak equivalence then so is $f$. In
  particular, any shrinkable morphism whose domain is either cofibrant
  or fibrant is a weak equivalence.
\end{proposition}

\proof
Suppose first that $d\cdot A$ is a weak equivalence.
We have commutative diagrams
\[ \xymatrix @C3pc {
    A \ar@<1ex>[r]^{d\cdot A} \ar@<-1ex>[r]_{c\cdot A} \ar@/^2pc/[rr]^{1}
    \ar@/_2pc/[rr]_{sf} & J\cdot A \ar[r]^h & A  &
     A \ar@<1ex>[r]^{d\cdot A} \ar@<-1ex>[r]_{c\cdot A} \ar@/^2pc/[rr]^{1}
    \ar@/_2pc/[rr]_{1} & J\cdot A \ar[r]^{e\cdot A} & A  
} \]
and since $d\cdot A$ is a weak equivalence it follows successively that
$e\cdot A$, $c\cdot A$, $h$, and finally $sf$ are so. But since $fs=1$ it
follows that $f$ (and $s$) are also weak equivalences.

Since $d$ is a trivial cofibration, if $A$ is cofibrant then $d\cdot
A$ will be a trivial cofibration, and in particular a weak
equivalence. 

The cases of $d\pitchfork A$ a weak equivalence, and of $A$ fibrant, are similar. \endproof

Using the preceding two propositions, we have the following result.
\begin{corollary}\label{cor:all-cofibrant}
If all objects of $\cc$ are cofibrant, then each trivial fibration is
\dsdr\ and each \dsdr\ morphism is a weak equivalence with a section.
\end{corollary}

\begin{remark}
  If all objects of \cv are fibrant and $I=1$, then in any interval
  $(J,d,c,e)$, the morphism $e\colon J\to I=1$ is not just a weak
  equivalence but a trivial fibration. Thus in this case any two
  intervals give the same notion of $\cv$-homotopy, and the \dsdr\ morphisms
  can be described using any one fixed interval.
\end{remark}

\begin{definition}
We write \ce for the collection of all \dsdr\ morphisms in \cv itself.
\end{definition}

\begin{examples}\label{ex:shrinking}
In each of our four main examples, all objects are cofibrant, and so
by Corollary~\ref{cor:all-cofibrant} every trivial fibration is
\dsdr, and every \dsdr\ morphism is a weak equivalence with a
section. In the first three of them, every weak equivalence with a
section is a trivial fibration, and so the \dsdr\ morphisms are
precisely the trivial fibrations. 
\begin{enumerate}
\item For $\cv$ equipped with the trivial model structure, the
  trivial fibrations and the \dsdr\ morphisms are both just the isomorphisms.
\item For $\cv$ with the split model structure, the trivial fibrations
  and the \dsdr\ morphisms are both just the split epimorphisms.
\item For $\Cat$ equipped with the canonical model structure, the trivial fibrations
  and the \dsdr\ morphisms are both just the surjective equivalences.
\item Now consider $\SSet$ equipped with the Joyal model structure.  By Corollary~\ref{cor:all-cofibrant}
the shrinkable morphisms lie between the weak equivalences and the trivial fibrations.  

 Not every \dsdr\ morphism is a trivial fibration. For example, let $J$ be the
  nerve of the free isomorphism, with its standard interval
  structure, then form the pushout of the two morphisms $1\to J$ to
  obtain a new interval $J'$, as in
  Proposition~\ref{prop:interval-for-transitivity}. The induced map
  $J'\to J$ is a \dsdr\  morphism 
 but is not a trivial fibration: in particular, $J$
  is fibrant while $J'$ is not, so $J'\to J$ cannot be a fibration.
  
 In general, all that we can say is that the \dsdr\ morphisms are weak
 equivalences with sections; however for a \dsdr\ morphism $f\colon X
 \to Y$ between quasicategories, the fibrant objects, we can say a
 little more.  First recall that an \emph{equivalence} of
 quasicategories is a map $f\colon X \to Y$ for which there exists a
 $g\colon Y \to X$ with $fg \cong_{J} 1_Y$ and $gf \cong_J 1_X$
 where $J$ is the nerve of the free isomorphism as discussed above.
 Let us call $f$ a \emph{surjective/retract equivalence} if in fact
 $fg = 1_Y$. Now since all objects are
 cofibrant in $\SSet$ each weak equivalence between quasicategories is
 a homotopy equivalence relative to any choice of interval; in
 particular, each \dsdr\ morphism $f\colon X \to Y$ between quasicategories is a surjective equivalence.
(Note, however, that we do not assert that the $(\cv,J)$-homotopy
$gf \cong_J 1_X$ is $f$-trivial --- indeed, while there is an
$f$-trivial $\cv$-homotopy relative to some interval, it does not seem that $f$-triviality is independent of the choice of interval.)
\end{enumerate}
\end{examples}

\begin{example}
  The category $\cgtop$ of compactly generated topological spaces is
  cartesian closed, and the standard model structure for topological
  spaces restricts to $\cgtop$ (see \cite[Section~2.4]{Hovey1999Model}
  for example). Here the notion of \dsdr\ morphism is dual to that of
  strong deformation retract, and it is in this case that the name
  ``shrinkable'' originated, as explained in the introduction. 
\end{example}

\begin{example}
  Another example is given by the monoidal model category $\Gray$: this is the
  category of (strict) 2-categories and 2-functors with the model
  structure of \cite{qm2cat}, as corrected in \cite{qmbicat}: the weak
  equivalences are the biequivalences, the trivial fibrations are the
  2-functors which are surjective on objects, and retract equivalences
  on homs. The monoidal structure is
  the Gray tensor product, as in \cite{GPS}. We shall see that
  in this case, the \dsdr\ morphisms are exactly the trivial
  fibrations which have a section.

  Not all objects are cofibrant, and not all
  trivial fibrations have sections, so not all trivial fibrations are
  \dsdr. On the other hand, all objects are fibrant, so we can use a
  standard interval for all $\cv$-homotopies; also all \dsdr\  morphisms are
  weak equivalences. If we use the free adjoint equivalence as our standard interval, a $\cv$-homotopy between $2$-functors in a pseudonatural equivalence. If $p\colon E\to B$
  is \dsdr, via $s\colon B\to E$ and a pseudonatural equivalence
  $sp\simeq 1$, then clearly $p$ is surjective on objects and an
  equivalence on hom-categories. Thus it will be a trivial fibration
  provided that it is full on 1-cells. For this, given $x,y\in E$ and
  a morphism $\beta\colon px\to py$ in $B$, we may compose with
  suitable components of the pseudonatural equivalence to get a
  morphism
  \[ \xymatrix{ x \ar[r]^{\sigma} & spx \ar[r]^{s\beta} & spy
      \ar[r]^{\sigma'} & y } \]
  and by $p$-triviality of $\sigma$, both $p\sigma$ and $p\sigma'$ are
  identities, and so $p$ is indeed full on 1-cells.

  Suppose conversely that $p\colon E\to B$ is a trivial fibration, and
  that $s\colon B\to E$ is a section of $p$. For each object $x\in E$,
  we have $pspx=px$ and so there is a morphism $\sigma_x\colon spx\to
  x$ with $p\sigma_x=1$ and similarly a morphism $\sigma'_x\colon x\to
  spx$, also over the identity. These form part of an equivalence,
  also over the identity. For any $\alpha\colon x\to y$ the square
  \[ \xymatrix{
      spx \ar[r]^{\sigma_x} \ar[d]_{sp\alpha} & x \ar[d]^{\alpha} \\
      spy \ar[r]_{\sigma_y} & y } \]
  need not commute, but both paths lie over the same morphism
  $p\alpha$ in $B$, and so there is a unique invertible 2-cell in the diagram,
  which is sent by $p$ to an identity. It now follows that $\sigma$
  determines an equivalence $sp\simeq 1$ over $B$, and so that $p$ is
  \dsdr.
\end{example}

\begin{example}
Let $R$ be a finite-dimensional cocommutative Hopf algebra, which is
Frobenius as a ring. Then the category of
modules over $R$ is a monoidal model category \cite[Section~2.2 and Proposition~4.2.15]{Hovey1999Model}.
All objects are cofibrant.  The trivial fibrations
are the surjections with projective kernel. In particular, the
codiagonal $R\oplus R\to R$  is a trivial fibration, and so $R\oplus
R$ itself is an interval. Thus the \dsdr\ morphisms are just the split
epimorphisms. Every trivial fibration is \dsdr, while the converse is
true (if and) only if every module is projective.
\end{example}

\begin{example}
  Let $R$ be a commutative ring, and $\Ch(R)$ the symmetric monoidal
  closed category of
  unbounded chain complexes of $R$-modules. This has a model structure \cite[Section~2.3]{Hovey1999Model}
  for which the fibrations are the pointwise surjections and the weak
  equivalences the quasi-isomorphisms. A morphism $f\colon X\to Y$ is
  \dsdr\ when it has a section $s$ and a chain homotopy between $sf$
  and $1$ which is $f$-trivial. This is of course rather stronger than
  being a trivial fibration.
\end{example}

The paper \cite{LackRosicky2012} introduced the notion of $\ce$-weak
left adjoint for any class of morphisms $\ce$.  Here we consider the
instance of this concept obtained by taking the class $\ce$ of \dsdr\
morphisms, and use the prefix ``shrink-'' to indicate $\ce$-weak
notions in this case. 

\begin{definition}
Let $U\colon \cb \to \ca$ be a \cv-functor.  A 
shrink-reflection of $A \in \ca$ is a morphism $\eta_A\colon A \to UA^{\prime}$ with the property that for all $B \in \cb$ the induced morphism
 \begin{equation*}
  \xymatrix{
  \cb(A^{\prime},B) \ar[r]^-{U} & \ca(UA^{\prime},UB) \ar[r]^-{\ca(\eta,UB)} &
      \ca(A,UB)}
  \end{equation*}
  is shrinkable.  If each object $A \in \ca$ admits a
  shrink-reflection then we say that $U$ admits a left
  shrink-adjoint. 
\end{definition}

\begin{examples}
\begin{enumerate}
\item For $\cv$ with the trivial model structure, a
  $\cv$-functor $U\colon \cb \to \ca$ as above admits a left shrink-adjoint just when it admits a genuine left adjoint.
\item For $\cv$ with the split model structure, $U\colon \cb \to \ca$
  admits a left shrink-adjoint  if we have morphisms $\eta \colon A \to UA^{\prime}$ for each $A$ such that the induced $\cb(A^{\prime},B) \to \ca(A,UB)$ is a split epimorphism.  In the case that $\cv = \Set$, this recovers ordinary weakness --- that is, for each $f\colon A \to UB$ there exists $f^{\prime}\colon A^{\prime} \to B$ with $Uf^{\prime} \circ \eta_A = f$.  In this case, setting $FA = A^{\prime}$ for each $A \in \cb$ and $Ff = (\eta_{B} \circ f)^{\prime}\colon FA \to FB$ for $f\colon A \to B$ equips $F$ with the structure of a graph morphism, but there is no reason why it should preserve composition. 
\item In the $\Cat$ case each $\cb(A^{\prime},B) \to \ca(A,UB)$ is a
  retract equivalence.  In this setting we can form $F$ as in the
  previous example, and with a similar definition $\alpha \mapsto
  F\alpha$ on $2$-cells.  Then given $f\colon A_1 \to A_2$ and
  $g\colon A_2 \to A_3$, since $\cb(A^{\prime}_1,A^{\prime}_3) \to
  \ca(A_1,UA^{\prime}_3)$ is a retract equivalence, we obtain a unique
  invertible $2$-cell such that $F_{f,g}\colon Fg \circ Ff \cong F(g
  \circ f)$.  Continuing in this way, $F$ becomes a
  pseudofunctor and we obtain a biadjunction in the sense of
  \cite{Kelly-limits}.
\item For $\SSet$ equipped with the Joyal model structure, if
  $\cb$ and $\ca$ are \emph{locally fibrant} --- that is, if they are enriched in quasicategories --- then each
  $\cb(A^{\prime},B) \to \ca(A,UB)$ is a retract equivalence of
  quasicategories.  As before, we can define $F$ as a graph morphism
  as in the preceding two examples.   Then given $f\colon A_1 \to A_2$
  and $g\colon A_2 \to A_3$ let us write
  $k\colon\cb(A^{\prime}_1,A^{\prime}_3) \to \ca(A_1,UA^{\prime}_3)$
  for the retract equivalence, with section $s$, and write $a = F(g
  \circ f)$ and $b=Fg \circ Ff$.  Then the $(\SSet,J)$-homotopies $1 \cong sk$
  and $sk \cong 1$ give a composite homotopy
  \[ F(g\circ f) = a \cong ska =skb \cong b = Fg \circ Ff \]
  in the hom-quasicategory $\cb(A^{\prime}_1,A^{\prime}_3)$,
  suggesting that $F$ could be turned into a pseudomorphism of
  simplicially enriched categories given a suitable formalisation of
  that concept.
 \end{enumerate}
\end{examples}

\section{Comma categories and slice categories}\label{sect:comma}

The proof of Freyd's general adjoint functor theorem given in
\cite{CWM} uses the fact 
that in order to construct a left adjoint to $U\colon\cb\to\ca$, it
suffices to construct an initial object in each comma category $A/U$,
for $A\in\ca$. Indeed, the universal property of the initial object
$\eta_A\colon A\to UA'$ implies that the induced map 
\begin{equation*}\label{eq:local}
\xymatrix @C3pc {
\cb(A^{\prime},B) \ar[r]^-{U} & \ca(UA^{\prime},UB) \ar[r]^-{\ca(\eta_{A},UB)} & \ca(A,UB)
}
\end{equation*}
is a \emph{bijection}, so that the maps $\eta_{A}\colon A \to
UA^{\prime}$ give the components of the unit for a left adjoint.  

In order to frame this in a manner suitable for generalisation,
observe that to say that the above function is a \emph{bijection} is
equally to say that it is a \emph{terminal object} in the slice category $\Set / \ca(A,UB)$ --- 
thus one passes from an initial object in the comma category to a terminal object in the slice.  

In the present section we develop the necessary results about comma
categories in the enriched context.  In our setting, we shall only be
able to construct \emph{bi-initial objects} in the comma category
$A/U$ and so are led to consider \emph{bi-terminal objects} in 
enriched slice categories, which turn out to be the \dsdr\
morphisms of the previous section. 

\subsection{Comma categories and slice categories}\label{subsect:comma}
Given $\cv$-functors $G\colon \cc \to \ca$ and $U\colon \cb \to \ca$
the enriched comma-category $G/U$ is the \emph{comma object}, below left

\begin{equation*}
\begin{xy}
(0,0)*+{G/U}="00";(20,0)*+{\cc}="10";(0,-15)*+{\cb}="01";(20,-15)*+{\ca}="11";
{\ar^{P} "00"; "10"};{\ar_{Q} "00"; "01"};{\ar^{G} "10"; "11"};{\ar_{U} "01"; "11"};
{\ar@{=>}^{\theta}(13,-6)*+{};(7,-12)*+{}};
\end{xy}
\hspace{1cm}
\xymatrix{
G/U \ar[d]_{\binom{P}{Q}} \ar[r] & \ca^{\two} \ar[d]^{\binom{\dom}{\cod}} \\
\cc\times\cb \ar[r]_-{G\times U} & \ca\x\ca
}
\end{equation*}
in the 2-category $\cv\cat$ of $\cv$-categories.  This has the universal
property that, given $\cv$-functors $R\colon \cd \to \cc$ and $S\colon \cd
\to \cb$, and a $\cv$-natural transformation $\lambda\colon G \circ R
\Rightarrow U \circ S$, there exists a unique $\cv$-functor $K\colon
\cd \to G/U$ such that $P \circ K = R$, $Q\circ K = S$
and $\theta \circ K =  \lambda$, as well as a 2-dimensional aspect
characterising \cv-natural transformations out of such a $K$.

For the reader unfamiliar with 2-dimensional limits, we point out that it is equally the pullback on the right above
in which $\ca^\two$ is the (enriched) arrow category, and $\binom{\dom}{\cod}$ the
projection which sends a morphism to its domain and codomain. 

The objects of $G/U$ consist of triples $(C, \alpha\colon GC \to
UB,B)$ where $C \in \cc$, $B \in \cb$, and $\alpha\colon GC \to UB$ is  in $\ca_0$; and with hom-objects as in the pullback below.
\begin{equation*}
\xymatrix{
G/U((C,\alpha,B),(C',\alpha',B')) \ar[dd] \ar[rr] && \cc(C,C') \ar[d]^{G} \\
&& \ca(GC,GC') \ar[d]^{\ca(1,\alpha^{\prime})} \\
\cb(B,B') \ar[r]_{U} & \ca(UB,UB') \ar[r]_{\ca(\alpha,1)} & \ca(GC,UB')}
\end{equation*}

A standard result of importance to us is the following one.

\begin{proposition}\label{prop:completeness}
Suppose that $\cb$ and $\cc$ admit $W$-limits for some weight
$W\colon\cd \to \cv$ and that $U\colon \cb \to \ca$ preserves them.
Then $G/U$ admits $W$-weighted limits for any $G\colon\cc\to\ca$, and they are preserved by the projections
$P\colon G/U \to \cc$ and $Q\colon G/U \to \cb$.
\end{proposition}
\begin{proof} 
Consider a diagram $T=(R,\lambda,S)\colon\cd \to G/U$ and let us write $$\cone{W}(X,T) := [\cd,\cv](W,G/U(X,T-))$$ for $X=(C,\alpha,B) \in G/U$.  We must prove that
 $\cone{W}(-,T)\colon G/U^{op} \to \cv$ is representable.  First observe that the definition of the hom-objects in $G/U$ gives us the components of a pullback
\begin{equation*}
\begin{xy}
(0,0)*+{G/U(X,T-)}="00";(85,0)*+{\cc(C,S-)}="20";
(85,-12)*+{\ca(GC,GS-)}="21";
(0,-24)*+{\cb(B,R-)}="02";(40,-24)*+{\ca(UB,UR-)}="12";(85,-24)*+{\ca(GC,UR-)}="22"; 
{\ar^{P} "00"; "20"};{\ar^{G} "20"; "21"};{\ar^{\lambda_{*}} "21"; "22"};{\ar_{Q} "00"; "02"};{\ar_-{U} "02"; "12"};{\ar_-{\alpha^{*}} "12"; "22"};
\end{xy}
\end{equation*}
in $[\cd,\cv]$.  Applying $[\cd,\cv](W,-)$, we obtain a pullback
\begin{equation}\label{eq:commapullback}
\begin{xy}
(0,0)*+{\cone{W}(X,T)}="00";(85,0)*+{\cone{W}(C,S)}="20";
(85,-12)*+{\cone{W}(GC,GS)}="21";
(0,-24)*+{\cone{W}(B,R)}="02";(40,-24)*+{\cone{W}(UB,UR)}="12";(85,-24)*+{\cone{W}(GC,UR)}="22"; 
{\ar^{P_{*}} "00"; "20"};{\ar^{G_{*}} "20"; "21"};{\ar^{\lambda_{*}} "21"; "22"};{\ar_{Q_{*}} "00"; "02"};{\ar_-{U_{*}} "02"; "12"};{\ar_-{\alpha^{*}} "12"; "22"};
\end{xy}
\end{equation}
in $\cv$.  Since the limits $\{W,S\}$ and $\{W,R\}$ exist and since $U$ preserves the latter, this pullback square is isomorphic to
\begin{equation*}
\begin{xy}
(0,0)*+{\cone{W}(X,T)}="00";(85,0)*+{\cc(C,\{W,S\})}="20";
(85,-12)*+{\ca(GC,G\{W,S\})}="21";
(0,-24)*+{\cb(B,\{W,R\})}="02";(40,-24)*+{\ca(UB,U\{W,R\})}="12";(85,-24)*+{\ca(GC,U\{W,R\})}="22"; 
{\ar^{} "00"; "20"};{\ar^{G} "20"; "21"};{\ar^{\varphi_{*}} "21"; "22"};{\ar^{} "00"; "02"};{\ar_-{U} "02"; "12"};{\ar_-{\alpha^{*}} "12"; "22"};
\end{xy}
\end{equation*}
for $\varphi \colon G \{W, S\} \to U \{W, R\}$ the morphism induced by 
\begin{equation*}
\xymatrix{
W \ar[r]^-{} & \cc(\{W,S\},S-) \ar[r]^-{G} & \ca(G\{W,S\},GS-) \ar[r]^{\lambda_{*}} & \ca(G\{W,S\},UR-)
}
\end{equation*}
and the universal property of $U\{W,R\}$.  By definition, the pullback in this square is
$G/U((C,\alpha,B),(\{W,S\},\varphi,\{W,R\}))$.  Hence the object $(\{W,S\},\varphi,\{W,R\})$ represents $\cone{W}(-, T)$ and is therefore the limit $\{W, T\}$.
\end{proof}

Let $\ci$ denote the \emph{unit \cv-category}, which has a single object with object of endomorphisms the monoidal unit $I$ of $\cv$.  An object $A \in \ca$ then corresponds to a $\cv$-functor $A\colon \ci \to \ca$.  We denote by $A/U$ the comma object depicted below left 
\begin{equation*}
\begin{xy}
(0,0)*+{A/U}="00";(20,0)*+{\ci}="10";(0,-15)*+{\cb}="01";(20,-15)*+{\ca}="11";
{\ar^{ P} "00"; "10"};{\ar_{ Q} "00"; "01"};{\ar^{A} "10"; "11"};{\ar_{U} "01"; "11"};
{\ar@{=>}^{\theta}(13,-6)*+{};(7,-12)*+{}};
\end{xy}
\hspace{1cm}
\begin{xy}
(0,0)*+{\ca/A}="00";(20,0)*+{\ca}="10";(0,-15)*+{\ci}="01";(20,-15)*+{\ca}="11";
{\ar^{ P} "00"; "10"};{\ar_{ Q} "00"; "01"};{\ar^{1} "10"; "11"};{\ar_{A} "01"; "11"};
{\ar@{=>}^{\theta}(13,-6)*+{};(7,-12)*+{}};
\end{xy}
\end{equation*}
in which we shall consider bi-initial objects, and the \emph{enriched
  slice category} $\ca/A$ as above right in which we shall consider
bi-terminal ones.  We record for later use that if $(B,b\colon
B\to A)$ and $(C,c\colon C\to A)$ are objects of $\ca/A$, the
corresponding hom is given by the pullback
\begin{equation}
  \label{eq:slice}
  \xymatrix @R2pc @C2pc {
    (\ca/A)((B,b),(C,c)) \ar[r] \ar[d] & \ca(B,C) \ar[d]^{\ca(B,c)} \\
    I \ar[r]_-b & \ca(B,A).
    } 
\end{equation}

\begin{remark}\label{remark:problems}
If the unit object of \cv is also terminal, then $\ci$ is the terminal
\cv-category, and is complete and cocomplete. We may then use
Proposition~\ref{prop:completeness} to deduce that, if $\cb$ has
$W$-limits and $U\colon\cb\to\ca$ preserves them, then $A/U$ has
$W$-limits for any $A\in\ca$. But if $I\neq 1$, then $\ci$ is not complete; indeed, almost by definition, $\ci$ has a terminal object if and only if $I=1$.
\end{remark}

\begin{proposition}\label{prop:terminalslice}
Let $\ca$ be a $\cv$-category and $A$ an object of $\ca$. The identity morphism on $A$ is a terminal object in the slice $\cv$-category $\ca/A$ if and only if $I=1$.
\end{proposition}
\begin{proof}
For any $f\colon B\to A$ we have the pullback
\[ \xymatrix{
  (\ca/A)((B,f),(A,1_A)) \ar[r] \ar[d]  & \cb(B,A) \ar[d]^{\cb(B,1_A)} \\
  I \ar[r]_{f}  & \cb(B,A) } \]
  and since the pullback of an isomorphism is an isomorphism, we
  conclude that $(\ca/A)((B,f),(A,1_A)) \cong I$, which is terminal if
  and only if $I$ is so.
\end{proof}

\begin{example}
By the preceding result, $1\colon A \to A$ is never terminal in the $\cv$-category $\ca/A$ unless $I=1$.
For a concrete example, let $\cv=\Ab$, and consider $\Ab/A$ for a
given abelian group $A$. An object consists of a homomorphism $f\colon
B\to A$, and a morphism from $(B,f)$ to $(A,1_A)$ consists of a pair
$(h,n)$ where $h\colon B\to A$ is a homomorphism, $n\in\bbz$, and
$h=n.f$. There is such a morphism $(n.f,n)$ for any $n\in\bbz$, and
these are all distinct, thus $(A,1_A)$ is not terminal.
\end{example}

\begin{proposition}\label{prop:ER_homotopy_category}
Suppose that $I=1$, and let $U\colon\cb \to \ca$ and $A \in \ca$.  
\begin{enumerate}
\item Then $A/U$ has underlying category $A/U_0$.  Furthermore, given parallel morphisms in $A/U$
\begin{equation*}
\xymatrix{
& A \ar[dl]_{b} \ar[dr]^{c} \\
UB \ar@<0.7ex>[rr]^{Uf} \ar@<-0.7ex>[rr]_{Ug} && UC}
\end{equation*}
a $\cv$-homotopy $f \cong g$ in $A/U$ is a $\cv$-homotopy $h\colon f \cong g$ in $\cb$ such that $Uh \circ b$ is trivial, which we will refer to as a $\cv$-homotopy under $A$.
\item Then $\ca/A$ has underlying category $\ca_0/A$.  Furthermore, given parallel morphisms in $\ca/A$
\begin{equation*}
\xymatrix{
B  \ar[dr]_{b} \ar@<0.7ex>[rr]^{f} \ar@<-0.7ex>[rr]_{g} && C \ar[dl]^{c} \\
& A}
\end{equation*}
a $\cv$-homotopy $f \cong g$ in $\ca/A$ is a $\cv$-homotopy $h\colon f \cong g$ in $\ca$ such that $c \circ h$ is trivial, which we will refer to as a $\cv$-homotopy over $A$.
\end{enumerate}
\end{proposition}
\begin{proof}
A morphism in $(A/U)_0$ is given by a map into the pullback as in 
\[\xymatrix{
    I \ar@{.>}[dr] \ar@/_1.5pc/[dddr]_{f'=1} \ar@/^1.5pc/[drr]^f \\
    & (A/U)((B,b),(C,c)) \ar[r] \ar[dd] & \cb(B,C) \ar[d]^-U \\
    && \ca(UB,UC) \ar[d]^{b^*} \\
 & I \ar[r]_{c} &  \ca(A,UC)
}\]
where the map $f'\colon I\to I$ is necessarily the identity by virtue of the assumption that
$I = 1$.  Therefore it amounts to a morphism $f\colon B \to C$ in
$\cb_0$ such that $Uf \circ b = c$: that is, a morphism of $A/U_0$.  Similarly, for $J$ an interval, a $(\cv,J)$-homotopy $f \cong g$ in $A/U$ is specified by a map from $J$ into the pullback whose component in $I$ is fixed as the structure map for the interval $e\colon J \to I$, again by the assumption that $I = 1$.  The component in $\cb(B,C)$ then specifies a $(\cv,J)$-homotopy $h\colon f \cong g$ for which $Uh \circ b$ is trivial.  

This analysis applies to the special case $A/\ca$ and so, by duality, establishes the corresponding claim for $\ca/A$.
\end{proof}

\begin{proposition}\label{prop:dsdrBiterminal}
If $I=1$ then $f \colon A \to B$ is \dsdr\ if and only if it is bi-terminal in $\cc/B$.
\end{proposition}
\begin{proof}
By Proposition~\ref{prop:terminalslice} the identity $1_B$ is terminal in $\cc/B$.  Therefore $f \colon A \to B$ is bi-terminal if and only if the unique map $f\colon (A,f) \to (B,1_B)$ is an equivalence.  Now a morphism $s\colon (B,1_B) \to
(A,f)$ is specified by a section $s\colon B \to A$ of $f$ so that we have $f \circ s = 1_B\colon (B,1_B) \to (B,1_B)$.  By Proposition~\ref{prop:ER_homotopy_category} a $\cv$-homotopy $s \circ f \cong 1_A$ in $\cc/B$ is specified by a $\cv$-homotopy $h\colon s \circ f \cong 1_A$ in $\cc$ such that $f \circ h$ is trivial, as required. 
\end{proof}

 In the case $I$ is not terminal, the above proposition need not hold.  We deal with this, and other issues related to a non-terminal $I$, by passing from \cv to the slice category $\cv/I$, which has an induced
monoidal structure for which the unit is terminal. Moreover, the comma $\cv$-categories $A/U$ and $\ca/A$ naturally give rise to $\cv/I$-categories.
See Section~\ref{sect:general} below.

\begin{theorem}\label{theorem:biterm}
Suppose that $I=1$.  Let \cb be a \cv-category with powers, and $U\colon\cb\to\ca$ a
  \cv-functor which preserves powers; if $\eta\colon A\to
  UA'$ is bi-initial in $A/U$ then 
  \begin{equation}\label{eq:value}
 \xymatrix{
      \cb(A^{\prime},B) \ar[r]^-{U} & \ca(UA',UB) \ar[r]^-{\ca(\eta,UB)} &
      \ca(A,UB) }
   \end{equation}
  is bi-terminal in $\cv/ \ca(A,UB)$.
\end{theorem}
\begin{proof}
By Proposition~\ref{prop:ER_homotopy_category} we need to find, for any object $\alpha$ of $\cv/\ca(A,UB)$ a morphism $f$ of $\cv/\ca(A,UB)$ as below
\[\xymatrix{
X \ar[rd]_-\alpha \ar@{-->}[rr]^-f & & \cb(A', B) \ar[ld] \\
& \ca(A, UB)
}\]
and for any two such morphisms a $\cv$-homotopy between them over $\ca(A, UB)$. 

Using the universal property of the power $X \pitchfork B$ and the fact that $U(X \pitchfork B) \cong X \pitchfork UB$, the above problem is translated to $A/U$ as below.
\[\xymatrix{
& A \ar[ld]_-\eta \ar[rd]^-{\alpha^\sharp} \\
UA' \ar@{-->}[rr]_-{Uf^\sharp} & & U(X \pitchfork B) \\
A' \ar@{-->}[rr]_-{f^\sharp} & & X \pitchfork B
}\]
Since $\eta$ is bi-initial, the required morphism $f^{\sharp}$ and so $f$ exists. If $f$ and $g$ are two such morphisms then we can, by bi-initiality of $\eta$, find a $\cv$-homotopy between $f^\sharp$ and $g^\sharp$ under $A$; writing it as
\[h^\sharp \colon A' \to J \pitchfork (X \pitchfork B) = X \pitchfork (J \pitchfork B),\]
we obtain $h \colon X \to \cb(A', J \pitchfork B) \cong J \pitchfork \cb(A', B)$ and this is the required $\cv$-homotopy between $f$ and $g$ over $\ca(A,UB)$
\end{proof}

\begin{remark}
A more conceptual explanation for the preceding result is as follows.  Given $U\colon\cb\to\ca$ and objects $B \in \cb$ and $A \in \ca$, we obtain a $\cv$-functor $K\colon (A/U)^{op} \to \cv/\ca(A,UB)$ sending $\eta\colon A \to UA'$ to the morphism \eqref{eq:value}.  If $\cb$ has powers and $U$ preserves them, then $K$ has a left $\cv$-adjoint.  The preceding result then follows immediately.
\end{remark}

\section{Enough cofibrantly-weighted limits}\label{sect:initial} 

In this section we describe the completeness and continuity conditions
which will arise in our adjoint functor theorem.
We start by recalling the notion of cofibrantly-weighted limit.

\subsection{Cofibrantly-weighted limits}

Let \cd be a small \cv-category, and consider the \cv-category
$[\cd,\cv]$ of \cv-functors from \cd to \cv. The morphisms in (the
underlying ordinary category of) $[\cd,\cv]$ are the \cv-natural
transformations. We shall say that such a \cv-natural transformation
$p\colon F\to G$ is a {\em trivial fibration} if it is so in the
pointwise/levelwise sense: that is, if each component $pD\colon FD\to
GD$ is a trivial fibration in \cv. 

We say that an object $Q\in[\cd,\cv]$ is {\em cofibrant} if, for each
trivial fibration $p\colon F\to G$ in $[\cd,\cv]$, the induced morphism
\[ [\cd,\cv](Q,p)\colon [\cd,\cv](Q,F)\to [\cd,\cv](Q,G) \]
is a trivial fibration in \cv.
In many cases, the projective model structure on $[\cd,\cv]$ will
exist and make $[\cd,\cv]$ into a model \cv-category, and then these
notions of cofibrant object and trivial fibration will be the ones that apply there. 

\begin{example}\label{ex:representable}
  Each representable $\cd(D,-)$ is cofibrant: this follows immediately
  from the Yoneda lemma, since $[\cd,\cv](\cd(D,-),p)$ is just
  $pD\colon FD\to GD$.
\end{example}

\begin{proposition}\label{prop:coprodscopowers}
  The cofibrant objects in $[\cd,\cv]$ are closed under coproducts and
  under copowers by cofibrant objects of \cv. 
\end{proposition}

\proof
The case of coproducts is well-known. As for copowers, let
$Q\colon\cd\to\cv$ and $X\in\cv$ be cofibrant. If $p\colon F\to G$ is
a trivial fibration in $[\cd,\cv]$ then $[\cd,\cv](X\cdot Q,p)$ is (up
to isomorphism) given by $[X,[\cd,\cv](Q,p)]$. Since $Q$ is cofibrant
and $p$ is a trivial fibration, it follows that $[\cd,\cv](Q,p)$ is a
trivial fibration; and now since also $X$ is cofibrant it follows that
$[X,[\cd,\cv](Q,p)]$ is also a trivial fibration.
\endproof

Many of the \cv-categories to which we shall apply our adjoint
functor theorem have all cofibrantly-weighted limits, but the
hypotheses of the theorem will be weaker than this. There are several
reasons for this, and in particular for not requiring idempotents to
split. For instance, the weak adjoint functor theorem of Kainen
\cite{Kainen1971Weak} which we wish to generalise does not assume
them.  Furthermore, in $2$-category theory there exist many natural
examples of $2$-categories admitting enough (but not all)
cofibrantly-weighted limits.  An example is the $2$-category of
\emph{strict }monoidal categories and \emph{strong} monoidal
functors: see \cite[Section~6.2]{Bourke2019Accessible}.
We start with the case where $I=1$. 

\subsection{Enough cofibrantly-weighted limits (when $I=1$)}\label{sect:pseudolimits}

We now define our key completeness and continuity conditions in
the case $I=1$; these will be modified later to deal with the case
where $I\neq 1$.

\begin{definition}
  Suppose that the unit object $I$ of \cv is terminal.
  We say that a \cv-category \cb has {\em enough cofibrantly-weighted limits} if, for any
  small \cv-category \cd, there is a chosen cofibrant weight $Q\colon
  \cd\to\cv$ for which the unique map $Q\to 1$ is a trivial fibration
  and for which \cb has $Q$-weighted limits. Similarly, a \cv-functor
  $U\colon\cb\to\ca$ preserves enough cofibrantly-weighted limits
  if, for each small $\cd$, there is such a $Q$ for which $U$
  preserves $Q$-weighted limits.
\end{definition}

In particular, \cb will have enough cofibrantly-weighted limits if it has all
cofibrantly-weighted limits.

\begin{remark} 
In our definition of ``enough cofibrantly-weighted limits'', we have asked that there be a
chosen fixed weight $Q\colon\cd\to\cv$ for each small \cd. But in
fact it is possible to allow $Q$ to depend upon the particular diagram
$\cd\to\cb$ of which one wishes to form the (weighted) limit. The cost
of doing this is that it becomes more complicated to express what it
means to preserve enough cofibrantly-weighted limits, but our main results can still be
proved with essentially unchanged proofs.
\end{remark}

\begin{example}
  If $\cv$ has the trivial model structure, then the trivial
  fibrations in $[\cd,\cv]$ are just the isomorphisms, and all objects
  are cofibrant. Thus in this case a \cv-category has enough
  cofibrantly-weighted limits if and only if it has all (unweighted)
  limits of $\cv$-functors (with small domain).
\end{example}

\begin{example}\label{ex:split-enough}
  If $\cv$ has the split model structure, then a \cv-category will
  have enough cofibrantly-weighted limits provided it has products.
  To see this, let \cd be a small \cv-category and consider the
  terminal weight $1\colon\cd\to\cv$.  Then
  $\sum_{D\in\cd}\cd(D,-)$ is a coproduct of representables, and so is
  cofibrant. The unique map $q\colon\sum_{D\in\cd}\cd(D,-)\to 1$ has
  component at $C\in\cd$ given by $\sum_{D\in\cd}\cd(D,C)\to 1$, which
  has a section picking out the identity $1\to \cd(C,C)$. Thus $q$ is
  a pointwise split epimorphism, and so a trivial fibration in
  $[\cd,\cv]$. And the $\sum_D\cd(D,-)$-weighted limit of
  $S\colon\cd\to \cb$ is just $\prod_{D\in\cd}SD$.
\end{example}

\begin{example}\label{ex:enough-Cat}
  In the case where $\cv=\Cat$, the enriched projective model
  structure on $[\cd,\Cat]$ exists and the cofibrant weights are
  precisely the flexible ones in the sense of \cite{Bird1989Flexible}:
  see Theorem~5.5 and Section~6 of \cite{Lack2007Homotopy-theoretic}.
  Thus a 2-category will have all cofibrantly-weighted limits just
  when it has all flexible limits. But it will have enough
  cofibrantly-weighted limits provided that it has all
  pseudolimits, and so in particular if it has products, inserters,
  and equifiers; in other words, if it has PIE limits in the sense of
  \cite{PIE}. Once again, the 2-category of strict monoidal categories
  and strong monoidal functors is an example which has PIE limits but
  not flexible ones: see \cite[Section~6.2]{Bourke2019Accessible}.
\end{example}

\begin{example}\label{ex:enough-SSet}
  In the case of $\SSet$, the enriched projective model structure on
  $[\cd,\SSet]$ exists (see Proposition~A.3.3.2 and Remark~A.3.3.4 of
  \cite{Lurie}) and has generating cofibrations
  $$\ci = \{\partial \Delta^n\cdot\cd(X,-) \to \Delta^n\cdot
  \cd(X,-)\colon n \in \mathbb N, X \in \cd\}$$ obtained by copowering
  the boundary inclusions of simplices by representables.  The
  $\ci$-cellular weights are, in this context, what Riehl and Verity
  call \emph{flexible weights} --- in particular, these are certain
  cofibrant weights.  Since Quillen's small object argument applied to
  the set $\ci$ produces a cellular cofibrant replacement of each
  weight, a simplicially enriched category will have enough
  cofibrantly-weighted limits provided that it has all flexible limits
  --- that is, those weighted by flexible weights.  In particular,
  each $\infty$-cosmos in the sense of \cite{Riehl2019Elements} admits
  enough cofibrantly-weighted limits.
\end{example}

\begin{lemma}\label{lem:biterminal}
  Let $Q\colon\cd\to\cv$ be a cofibrant weight for which the unique
  map $Q\to 1$ is a trivial fibration. Then $Q$ is bi-terminal in the
  full subcategory of $[\cd,\cv]$ consisting of the cofibrant weights.
\end{lemma}
\begin{proof}
  Let $G$ be cofibrant. Then the unique map from $[\cd,\cv](G,Q)$ to
  $[\cd,\cv](G,1)$ is a trivial fibration. Since the left vertical in
  each of the following diagrams is a cofibration in \cv
  \[\xymatrix{
      \varnothing \ar[d] \ar[rr]^-{} && [\cd,\cv](G,Q) \ar[d] \\
      I \ar@{.>}[urr] \ar[rr] && 1 = [\cd,\cv](G,1) }
    \hspace{1cm} \xymatrix{
      I + I \ar[d] \ar[rr]^-{} && [\cd,\cv](G,Q) \ar[d] \\
      J \ar@{.>}[urr] \ar[rr] && 1= [\cd,\cv](G,1) }
  \]
  it follows that each diagram has a filler. This implies the
  existence and essential uniqueness of maps $G\to Q$.
\end{proof}

Consider a diagram $J\colon \cd \to \cb$, and let $L=\{Q,J\}$.
Now $Q$ is bi-terminal with respect to cofibrant objects by
Lemma~\ref{lem:biterminal}, and each representable is cofibrant by
Example~\ref{ex:representable}, so that there exists a
morphism $s_D\colon \cd(D,-) \to Q$.  Since weighted limits are
(contravariantly) functorial in their weights, we obtain a morphism
\[p_D\colon L = \{Q,J\} \to \{\cd(D,-),J\} = JD.\] By bi-terminality
of $Q$ once again, for any $f \colon C \to D$ in $\cd$ there exists a
$\cv$-homotopy in the triangle below left
$$\xy
(0,0)*+{\cd(C,-)}="a0"; (-15,-15)*+{\cd(D,-)}="b0";(15,-15)*+{Q}="c0";
{\ar^{\cd(f,-)} "b0"; "a0"}; {\ar^{s_{C}} "a0"; "c0"}; {\ar_{s_{D}}
  "b0"; "c0"}; (1,-9)*{{\cong}_{}}; \endxy \hspace{2cm} \xy
(-30,0)*+{L}="c0";(0,0)*+{JC}="a0"; (-15,-15)*+{JD}="b0"; {\ar^{Jf}
  "a0"; "b0"}; {\ar^{p_{C}} "c0"; "a0"}; {\ar_{p_{D}} "c0"; "b0"};
(-15,-8)*{{\cong}_{}}; \endxy$$ and so, on taking limits of $J$ by the
given weights, a $\cv$-homotopy in the triangle on the right.

\begin{lemma}\label{lemma:lukas}
  If $J\colon\cd\to\cb$ is fully faithful, and $L=\{Q,J\}$ as above,
  then the projection $p_D\colon L\to JD$ satisfies
  $f\circ p_D\cong 1_L$ for any morphism $f\colon JD\to L$.
\end{lemma}
\begin{proof}
  Using fully faithfulness of $J$ we define $k\colon Q\to \cd(D,-)$ as
  the unique morphism rendering commutative the upper square below
  \[ \xymatrix{
      Q \ar[r]^-{\eta} \ar[d]_k & \cb(L,J-) \ar[d]^{\cb(f,J-)} \\
      \cd(D,-) \ar[r]^-J_-{\cong} \ar[d]_{s_D} & \cb(JD,J-) \ar[d]^{\cb(p_D,J-)} \\
      Q \ar[r]_-{\eta} & \cb(L,J-) } \] in which $\eta$ is the unit of
  the limit $L=\{Q,J\}$, and the lower commutes essentially by
  definition of $p_D$. Bi-terminality of $Q$ among cofibrant objects
  in $[\cd,\cv]$ implies that the composite $s_D\circ k$ is
  $\cv$-homotopic to the identity.

  Consider the weighted limit \cv-functor
  $\{-,J\}\colon {[\cd,\cv]\op_\cb}\to \cb$, where the domain is
  restricted to the full subcategory containing those weights $W$ for
  which $\cb$ admits all $W$-weighted limits.  Since $[\cd,\cv]_\cb$
  contains $Q$ and $\cd(D,-)$ we use the fact that \cv-functors
  preserve $\cv$-homotopies to deduce that
  $\{k,J\}\circ \{s_D,J\}\cong 1$; in other words that
  $f\circ p_D\cong 1$.
\end{proof}

\begin{example}\label{ex:isoinserter} 
Let \Eq be the free parallel pair, involving maps
$f,g\colon A\to B$, made into a free \cv-category. Consider the
constant functor $\Delta I\colon \Eq\to\cv$, which is equally the terminal weight.
We first show how a
(standard) interval gives a cofibrant replacement of $\Delta I$ and
interpret the corresponding limit.  Then we relate this special type
of cofibrant replacement to a general one.
	
Given a standard interval --- that is, an interval 
\[ \xymatrix @R2pc @C2pc {
    I \ar@<1ex>[r]^d \ar@<-1ex>[r]_c & J \ar[r]^e & I
  } \]
 in which $e$ is a trivial fibration, one can show that the pair $d$, $c$ constitutes a cofibrant weight $Q \colon \Eq \to \cv$ and admits a pointwise trivial fibration
\[  \xymatrix @R2pc @C2pc {
    I \ar@<-1ex>[r]_-c \ar@<1ex>[r]^-d \ar[d]_-1 &  J \ar[d]^e\\
    I \ar@<-1ex>[r]_-1 \ar@<1ex>[r]^-1 &  I
  } \]
to the constant functor $\Delta I$. A $Q$-cone from an object
$X$ to a diagram $f ,\, g \colon A \rightrightarrows B$
amounts to a natural transformation
	\[ \xymatrix @R2pc @C2pc {
      I \ar@<-1ex>[d]_-d \ar@<1ex>[d]^-c \ar[r]^-k & \cc(X, A) \ar@<-1ex>[d]_-{f_*} \ar@<1ex>[d]^-{g_*} \\
      J \ar[r]_-h & \cc(X, B)
  } \]	
or equivalently a morphism $k \colon X \to A$ 
equipped with a homotopy $h \colon fk \cong gk$.
Thus, a limit $\{Q, F\}$ possesses a universal such pair and is thus
what one might call a {\em $Q$-isoinserter}, or
{\em isoinserter} if $Q$ is understood. 

Next, suppose that $Q'\colon \Eq\to\cv$ is cofibrant and
$q'\colon Q'\to \Delta I$ is a pointwise trivial fibration, as in
\[  \xymatrix @R2pc @C2pc {
    I' \ar@<1ex>[r]^{d'} \ar@<-1ex>[r]_{c'} \ar[d]_{a} & J' \ar[d]^b \\
    I \ar@<1ex>[r]^1 \ar@<-1ex>[r]_1 & I 
  } \]
and let  $u$ be a section of the trivial fibration $a$. Factorise the
induced map $(d'u~c'u)$ as a cofibration followed by a trivial
fibration as in
\[ \xymatrix @R2pc @C2pc {
    I+I \ar[r]^{u+u} \ar[d]_{(d~c)} & I'+I' \ar[d]^{(d'~c')} \\
    J \ar[r]^{v}  & J' 
  } \]
to give an interval
\[ \xymatrix @R2pc @C2pc {
I \ar@<1ex>[r]^{d} \ar@<-1ex>[r]_{c} & J \ar[r]^{bv} & I.
} \]
together with a natural transformation
\[  \xymatrix @R2pc @C2pc {
    I \ar@<1ex>[r]^{d} \ar@<-1ex>[r]_{c} \ar[d]_{u} & J \ar[d]^v \\
    I' \ar@<1ex>[r]^{d'} \ar@<-1ex>[r]_{c'} & J'. 
  } \]
Thus, any $Q'$-cone gives a $Q$-cone and, in particular, a morphism $k \colon X \to A$ and a homotopy $h \colon fk \cong gk$.
We call the limit $\{Q',F\}$, the {\em $Q'$-isoinserter} of $f$ and $g$, or
just the {\em isoinserter} if $Q'$ is understood.
\end{example}

\subsection{Enough cofibrantly-weighted limits ($I\neq 1$)}

For $\ci$ the unit $\cv$-category we write $I\colon \ci \to \cv$ for the $\cv$-functor selecting the object $I$.

\begin{definition}\label{def:enoughGeneral}
  We say that a $\cv$-category \cb has {\em enough cofibrantly-weighted limits} if, for each small \cv-category \cd and each
  \cv-functor $P\colon \cd\to\ci$, there is a cofibrant \cv-weight
  $Q\colon \cd \to\cv$ and pointwise trivial fibration $Q\to IP$ for
  which \cb has $Q$-weighted limits. Similarly, $U\colon \cb\to \ca$
  preserves enough cofibrantly-weighted limits if, for each small
  $\cd$ and each $P$, there is such a $Q$ for which $U$ preserves $Q$-limits.
\end{definition}

In the case where $I=1$, the $\cv$-category $\ci$ is terminal,
and so any $\cv$-category $\cd$ has a unique such $P$; furthermore,
the composite $IP$ is then the terminal object of $[\cd,\cv]$, and
so this does agree with our earlier definition. 

\begin{example}
  For the trivial model structure on $\cv$, no longer assuming that
  $I$ is terminal, a $\cv$-category $\ck$ will have enough cofibrantly-weighted
  limits if and only if it has all limits weighted by functors of the
  form $IP\colon\cd\to\cv$ for a small $\cv$-category $\cd$, a
  $\cv$-functor $P\colon\cd\to\ci$, and $I\colon\ci\to\cv$ the weight
  picking out the object $I\in\cv$.
\end{example}

\begin{example}
 For the split model structure on $\cv$, no longer assuming
 that $I$ is terminal, each weight $W \in [\cd,\cv]$ admits a
 canonical cofibrant replacement $ q\colon W^{\prime} \to W$ given by the
 evaluation map
 \begin{equation*}
   \xymatrix{
     {}\sum_{D \in \cd}WD\cdot\cd(D,-) \ar[r]^-q & W
   }
 \end{equation*}
 To see this, observe that the left hand side is a coproduct of
 copowers of representables, and so cofibrant by Proposition~\ref{prop:coprodscopowers}.  Furthermore, each
 component of $q$ is a split epimorphism with section
 \[ WA \cong WA\cdot I \to WA\cdot\cd(A,A) \to \Sigma_{D \in
     \cd}WA\cdot\cd(D,A) \] obtained by first copowering the map
 $I \to \cd(A,A)$ defining the identity morphism on $A$, and then
composing this with the coproduct inclusion.  Thus $q$ is a
pointwise split epimorphism, and so a trivial fibration in
$[\cd,\cv]$.  In the case that $W = IP$,  as in Definition~\ref{def:enoughGeneral},
each $WA$ is just $I$.   Thus the right hand side above reduces to
$\sum_{D \in \cd}\cd(D,-)$, 
whence a $\cv$-category has enough cofibrantly-weighted limits provided it
has products.
\end{example}

\section{The adjoint functor theorem in the case
  $I=1$}\label{sect:I=1}

The key technical step in the proof of Freyd's general adjoint functor theorem is to prove that
a complete category with a weakly initial set of objects has an
initial object.
We shall now show that, on replacing completeness by homotopical completeness, we can still construct a bi-initial object.
It is then straightforward to deduce the adjoint functor theorem.

\begin{theorem}\label{thm:biinitial}
  Suppose that the unit object $I$ of $\cv$ is terminal.  
  If \cb is a \cv-category with enough cofibrantly-weighted limits,
  and $\cb_0$ has a weakly initial set of objects, then \cb has a bi-initial object. 
\end{theorem}

\proof
Let \cg be the full subcategory of \cb consisting of the objects
appearing in the weakly initial set.
We write $J\colon \cg\to\cb$ for the inclusion of this small full
subcategory.
By assumption, \cb admits $Q$-weighted limits for some cofibrant
weight $Q\colon \cg\to\cv$ for which the unique map $Q\to 1$ is a trivial fibration.
We shall show that the limit $L=\{Q,J\}$ is bi-initial in $\cb$.

Following the notation of Section~\ref{sect:pseudolimits}, we obtain morphisms $p_C\colon L \to JC$ for each $C \in \cg$, such that for $f\colon C \to D$ the triangle

$$\xy
(-30,0)*+{L}="c0";(0,0)*+{JC}="a0"; (-15,-15)*+{JD}="b0";
{\ar^{Jf} "a0"; "b0"}; 
{\ar^{p_{C}} "c0"; "a0"}; 
{\ar_{p_{D}} "c0"; "b0"}; 
(-15,-8)*{{\cong}_{}};
\endxy$$
commutes up to $\cv$-homotopy.

Let $B \in \cb$.  Since \cg is weakly initial,  there exists a
$C\in\cg$ and a morphism $f\colon JC\to B$, and now $f\circ p_C\colon
L\to B$ gives the existence part of the required property of $L$.

Suppose now that $B\in\cb$ and that $f,g\colon L\to B$ are two
morphisms.
As in Example~\ref{ex:isoinserter}, we may form the isoinserter
$k\colon K\to L$ of $f$ and $g$, and $fk\cong gk$. 
By weak initiality of $\cg$ there is a morphism $u\colon JC\to K$ for
some $C\in\cg$. Now
\[ f\circ k\circ u\circ p_C \cong g\circ k\circ u\circ p_C \]
and so it will suffice to show that $k\circ u\circ p_C\cong 1$.
Lemma~\ref{lemma:lukas} gives $v\circ p_C\cong 1$ for any $v\colon JC\to L$ and, in particular, for
$v=k\circ u$. 
\endproof

\begin{theorem}\label{thm:adjunction1}
Suppose that the unit object $I$ of \cv is terminal.
 Let \cb be a \cv-category with powers and enough cofibrantly-weighted limits, and let
  $U\colon\cb\to\ca$ be a \cv-functor that preserves them. Then $U$
  has a left shrink-adjoint if and only if $U_0$ satisfies the solution
  set condition.
\end{theorem}
\begin{proof}
Suppose first that $U$ has a left shrink-adjoint, and
$\eta_A\colon A\to UA'$ is a shrink-reflection.   Then the singleton
family consisting of $(A',\eta)$ is a solution set for $U_0$. This proves the ``only if'' direction.

Since $I$ is terminal the unit $\cv$-category $\ci$ is terminal in
$\cv\cat$ and hence complete as a $\cv$-category.  Then for any $A\in\ca$,
it follows from Proposition~\ref{prop:completeness} that the comma
category $A/U$ has any type of limit which $\cb$ has and
$U\colon\cb\to\ca$ preserves, and so has enough cofibrantly-weighted limits.   
By Proposition~\ref{prop:ER_homotopy_category} we know that $(A/U)_0\cong A/U_0$, which has a weakly initial set of objects since $U_0$ satisfies the solution set condition. Hence,  by Theorem~\ref{thm:biinitial}, $A/U$
has a bi-initial object $\eta\colon A \to UA'$.  Since $\cb$ has
powers and $U$ preserves them, it follows from
Theorem~\ref{theorem:biterm} that 
 \begin{equation*}
  \xymatrix{
 \cb(A',B) \ar[r]^-{U} & \ca(UA',UB) \ar[r]^-{\ca(\eta,UB)} &
      \ca(A,UB)}
  \end{equation*}
  is bi-terminal in $\cv/ \ca(A,UB)$ --- in other words, by Proposition~\ref{prop:dsdrBiterminal}, it is \dsdr.
\end{proof}

\section{$\cv/I$-categories and the adjoint functor theorem in the case
  $I\neq1$}\label{sect:Ineq1}\label{sect:general}

The proof given in the previous section relied on the fact that the
enriched comma categories $A/U$ had enough cofibrantly-weighted
limits,
which we were able to prove only in the case where $I=1$, since only then could we be
sure that the unit \cv-category \ci had enough cofibrantly-weighted
limits.
We shall overcome this problem by viewing $A/U$ as a category enriched
in the monoidal category $\cv/I$, whose unit \emph{is} terminal.
To this end, observe that, since the unit $I$ is (trivially)  a commutative monoid in \cv, the
slice category {$\cv/I$} becomes a symmetric monoidal category, with
unit $1\colon I\to I$ and with the tensor product of
$(X,x\colon X\to I)$ and $(Y,y\colon Y\to I)$ given by
$(X\ox Y, x\ox y\colon X\ox Y\to I)$. In particular, the unit is indeed
terminal in {$\cv/I$}. The resulting symmetric monoidal category
$\cv/I$ is also closed, with the internal hom of $(Y,y)$ and $(Z,z)$
given by the left vertical in the pullback
\begin{equation}\label{eq:inthom}
 \xymatrix @R2pc @C2pc {
    {}\<(Y,y),(Z,z)\> \ar[rr] \ar[d] &&     [Y,Z] \ar[d]^{[Y,z]} \\
    I \ar[r]_j & [I,I] \ar[r]_{[y,I]} & [Y,I] }
    \end{equation}
 The category $\cv/I$ is once again complete and cocomplete. It becomes a
monoidal model category if we define a morphism to be a cofibration,
weak equivalence, or fibration just when its underlying morphism in
$\cv$ is one.

A category enriched in $\cv/I$ involves a collection of objects together with, for each pair $X,Y$ of objects, an object
$$
C_{X,Y}\colon\cc(X,Y) \to I
$$
of $\cv/\ci$.  Translating through the remaining structure, this is
equally to specify a $\cv$-category $\cc$ together with a functor
$C\colon\cc \to \ci$ to the unit $\cv$-category.  
We sometimes write $\en{C} = (\cc,C)$ for such a $\cv/I$-enriched category, and we refer to the functor $C$ as the {\em augmentation}. 

\begin{proposition}
This correspondence extends to an isomorphism of categories $\cv/I\cat\cong\cv\cat/\ci$.
\end{proposition}

\proof
A $\cv/I$-functor $F\colon\en{\cc}\to\en{\cd}$  consists of a $\cv$-functor
$F\colon\cc\to\cd$ for which the action $F_{C,D}\colon \cc(C,D)\to\cd(FC,FD)$ on
homs commutes with the maps into $I$; equivalently, such that $F$
commutes with the augmentations.\endproof

\begin{example}\label{ex:comma}
The augmentations $P\colon A/U \to \ci$ and $Q\colon\ca/A \to \ci$ equip the $\cv$-categories $A/U$ and $\ca/A$ with the structure of $\cv/I$-categories.
\end{example}

The next four results extend
Propositions~\ref{prop:terminalslice},
\ref{prop:ER_homotopy_category}, \ref{prop:dsdrBiterminal}, and
Theorem~\ref{theorem:biterm} respectively, with essentially the same proof as before in each case. 

\begin{proposition}\label{prop:terminalslice2}
 The identity morphism on $A$ is a terminal object in the slice $\cv/I$-category $\ca/A$.
\end{proposition}
\begin{proof}
From ~\eqref{eq:slice}, for any $f\colon B\to A$ we have the pullback
\[ \xymatrix{
  (\ca/A)((B,f),(A,1_A)) \ar[r] \ar[d]  & \cb(B,A) \ar[d]^{\cb(B,1_A)} \\
  I \ar[r]_{f}  & \cb(B,A) } \]
 and since the pullback of an isomorphism is an isomorphism, the left vertical map is an isomorphism over $I$.  As such, it is terminal as an object of $\cv/I$.
\end{proof}

\begin{proposition}\label{prop:ER_homotopy_category2}
Let $U\colon\cb \to \ca$ and $A \in \ca$.  
\begin{enumerate}
\item Then the $\cv/I$-category $A/U$ has underlying category $A/U_0$.  Furthermore, parallel morphisms $f,g\colon(B,b) \rightrightarrows (C,c)$ are $\cv/I$-homotopic if and only if $f$ and $g$ are $\cv$-homotopic under $A$.
\item Then $\ca/A$ has underlying category $\ca_0/A$.  Furthermore, parallel morphisms $f,g\colon(B,b) \rightrightarrows (C,c)$ are $\cv/I$-homotopic if and only if $f$ and $g$ are $\cv$-homotopic over $A$.
\end{enumerate}
\end{proposition}
\begin{proof}
Since the unit in $\cv/I$ is $1\colon I \to I$, a morphism in $(A/U)_0$ is given by a map into the pullback
\[\xymatrix{
   I \ar[r] \ar[ddr]_{1} & (A/U)((B,b),(C,c)) \ar[r] \ar[dd] & \cb(B,C) \ar[d]^-U \\
    && \ca(UB,UC) \ar[d]^{b^*} \\
 & I \ar[r]_{c} &  \ca(A,UC)
}\]
whose component in $I$ is the identity. 
As in the proof of Proposition~\ref{prop:ER_homotopy_category} this amounts to a morphism of $A/U_0$.  Now a $\cv/I$-interval is an interval in $\cv$ viewed as a diagram
\[ \xymatrix{
    (I+I,\nabla) \ar[r]^-{(d~c)} & (J,e) \ar[r]^e & (I,1)} \]
 in $\cv/I$.  Therefore a $(J,e)$-homotopy $f \cong g$ in $A/U$ is
 specified by a map from $J$ into the pullback whose component in $I$
 is fixed as $e\colon J \to I$, and so as in the proof of
 Proposition~\ref{prop:ER_homotopy_category} amounts to a
 $(\cv,J)$-homotopy $h\colon f \cong g$ for which $Uh \circ b$ is trivial.  The case of $\ca/A$ follows by duality, as before.
\end{proof}

\begin{proposition}\label{prop:shrinking2}
A morphism $f\colon A \to B \in \cc$ is shrinkable if and only if it is bi-terminal in the slice $\cv/\ci$-category $\cc/B$.
\end{proposition}
\begin{proof}
The proof of Proposition~\ref{prop:dsdrBiterminal} carries over unchanged on replacing the use of Propositions~\ref{prop:terminalslice} and \ref{prop:ER_homotopy_category} by their respective generalisations, Propositions~\ref{prop:terminalslice2} and \ref{prop:ER_homotopy_category2}.
\end{proof}

\begin{theorem}\label{theorem:biterm2}
Let \cb be a \cv-category with powers, and $U\colon\cb\to\ca$ a
  \cv-functor which preserves them; if $\eta\colon A\to
  UA'$ is bi-initial in the $\cv/I$-category $A/U$ then 
\[  \xymatrix @C3pc {
      \cb(A^{\prime},B) \ar[r]^-{U} & \ca(UA',UB) \ar[r]^-{\ca(\eta,UB)} &
      \ca(A,UB) } \] 
is bi-terminal in the $\cv/I$-category $\cv/ \ca(A,UB)$.
\end{theorem}
\begin{proof}
The proof of Theorem~\ref{theorem:biterm} carries over unchanged on replacing the use of Proposition~\ref{prop:ER_homotopy_category} by the more general Proposition~\ref{prop:ER_homotopy_category2}.
\end{proof}

We now turn to our main theorem. Its proof will rely on a technical
result, Proposition~\ref{cor:enough_cofibrant_limits_comma}, which the remainder of the section will be devoted to proving.

\begin{theorem}\label{thm:adjunction}
  Let \cb be a \cv-category with powers and enough cofibrantly-weighted limits, and let
  $U\colon\cb\to\ca$ be a \cv-functor that preserves them. Then $U$
  has a left shrink-adjoint if and only if $U_0$ satisfies the
  solution set condition.
\end{theorem}

\proof If $U$ has a left shrink-adjoint, and
$\eta_A\colon A\to UA'$ is a shrink-reflection, then the singleton
family consisting of $(A',\eta)$ is a solution set for $U_0$. This proves the ``only if'' direction.

Suppose conversely that $U_0$ satisfies the solution set condition.  By Proposition~\ref{cor:enough_cofibrant_limits_comma} below, it follows that the $\cv/I$-category $A/U$ has enough
cofibrantly-weighted limits. By Proposition~\ref{prop:ER_homotopy_category2} we know that the underlying category $(A/U)_0$ of the $\cv/I$-category $A/U$ is isomorphic to $A/U_0$. Since $U_0$ satisfies the solution set condition $A/U_0$, and so also $(A/U)_0$, have weakly initial sets of objects. 
Thus, by Theorem~\ref{thm:biinitial}, $A/U$ has a bi-initial object $\eta\colon A \to UA'$, and now by Theorem~\ref{theorem:biterm2} the morphism
  \[ \xymatrix @R1pc {
      \cb(A^{\prime},B) \ar[r]^-{U} & \ca(UA',UB) \ar[r]^-{\ca(\eta,UB)} &
      \ca(A,UB) }  . \]
is bi-terminal in the $\cv/I$-category $\cv/\ca(A,UB)$, and so is shrinkable by Proposition~\ref{prop:shrinking2}. 
\endproof

In order to prove Proposition~\ref{cor:enough_cofibrant_limits_comma},
we need to consider weighted limits in $\cv/I$-categories of the form
$P\colon A/U \to \ci$.  First, we show that these are comma categories in the $\cv/I$-enriched sense. In order to formulate this statement correctly, let us first observe that the forgetful functor $\cv\cat/\ci \to \cv\cat$ has a right adjoint $R$, sending $\ca$ to the $\cv/I$-category $\pi_2\colon\ca \times \ci \to \ci$.

\begin{lemma}\label{lem:comma2}
Given a $\cv$-functor $U\colon\cb \to \ca$ and an object $A \in \ca$, the $\cv
/I$-category $P\colon A/U \to \ci$ is isomorphic to the comma $\cv/I$-category $A / { R(U)}$.
\end{lemma}
\begin{proof}
By Proposition~\ref{prop:ER_homotopy_category2}, an object of the
$\cv/I$-category $A/U$ is given by a pair $(B,f\colon I \to \ca(A,UB)
\in \cv)$, which bijectively corresponds to a pair
$(B,(f,1)\colon(I,1) \to (\ca(A,UB) \times I,\pi_2) \in \cv/I)$; 
that is, to an object of $A/R(U)$.

Given objects $(B,(f,1))$ and $(C,(g,1))$ of $A / R(U)$, consider the following diagram, in which each region is a pullback.
\[ \xymatrix{
    (A/R(U))((B,(f,1)),(C,(g,1))) \ar[r]_{} \ar[dd]_{} & \cb(B,C)\x I \ar[r]^-{\pi_1}
    \ar[d]_{U\x 1} 
    & \cb(B,C) \ar[d]_{U} \\
    & \ca(UB,UC)\x I \ar[r]^-{\pi_1} \ar[d]_{\ca(f,UC) \x 1} & \ca(UB,UC) \ar[d]_{\ca(f,UC)} \\
    I \ar[r]^{(g,1)} \ar@/_1pc/[rr]_g & \ca(A,UC)\x I \ar[r]^-{\pi_1} &
    \ca(A,UC) } \]
We may deduce an isomorphism
\[ (A/R(U))((B,(f,1)),(C,(g,1))) \cong A/U((B,f),(C,g)) \]
in $\cv/I$,
since the left hand side is defined by the pullback on the left, and the right hand side is defined by the composite pullback.
The remaining verifications of functoriality are straightforward.
\end{proof}

Next, we will consider weighted limits in the $\cv/I$-enriched sense.   To get started, observe, on comparing ~\eqref{eq:slice} and ~\eqref{eq:inthom}, that $\cv/I$, as a $\cv/I$-enriched category, is given by the slice $\cv$-category $Q\colon\cv/I \to \ci$.  We denote it by $\en{V/I}=(\cv/I,Q)$.  

Let $\en{D} = (\cd,D)$ be a $\cv/I$-category.  Then a $\cv/I$-weight $\en{D} \to \en{V/I}$ is specified by a commutative triangle as on the left below.
\[\xymatrix{
\cd \ar[r] \ar[dr]_{D} & \cv/I \ar[r] \ar[d]_{Q} & \cv \ar@{=}[d] \ar@{}[ld]|-{\Downarrow} & {} \ar@{}[d]|-{\equiv} & \cd \ar[rr]^-W \ar[rd]_{D}  & {} \ar@{}[d]|-{\Downarrow w} & \cv \\
& \ci \ar[r]_-I & \cv & & & \ci \ar[ru]_-I
}\]
By the universal property of the slice $\cv$-category $\cv/\ci$ this
amounts to a $\cv$-weight $W\colon \cd \to \cv$ together with a $\cv$-natural transformation $w \colon W \Rightarrow ID$.  We write $\en{W}=(W, w)$ for the $\cv/I$-weight.

Building on the above description, it is straightforward to see that
the presheaf $\cv/I$-category $[\en{D},\en{V/I}]$ is given by the
slice $\cv$-category $[\cd,\cv]/ID$ equipped with its natural
projection to $\ci$.  In particular, its homs are defined by the
following pullbacks.

\begin{equation}\label{eq:hom}
\xymatrix{
[\en{D},\en{V/I}](\en{F},\en{G}) \ar[d]_{U} \ar[rr] && [\cd,\cv](F,G) \ar[d]^{g_{*}} \\
I \ar[rr]_{f} && [\cd,\cv](F,ID)
}
\end{equation}
Finally, observe that given a diagram $S\colon\en{D} \to \en{A}$ and an object $X \in \en{A}$ the induced presheaf $$\en{A}(X,S-)\colon\en{D} \to \en{V/I}$$ consists of the $\cv$-weight $\ca(X,S-)\colon\cd \to \cv$ together with the augmentation $\ca(X,S-) \to I D$ having components $A_{X,SY}\colon\ca(X,SY) \to I$.  

\begin{lemma}\label{lem:basechange}
Let $\en{W}=(W,w)$ be a $\cv/I$-weight.  If $\cb$ has $W$-weighted limits then $R(\cb)$ has $\en{W}$-weighted limits; if $U\colon\cb \to \ca$ preserves $W$-weighted limits then $R(U)$ preserves $\en{W}$-weighted limits.
\end{lemma}
\begin{proof}
Consider a diagram $\en{D} \to R(\cb) = (\cb \times \ci,\pi_2)$, which
is necessarily of the form $\overline{S}=(S,D)$ for $S\colon\cd \to
\cb$.  We claim that
the limit $\{W,S\}$ also has the universal property of
$\{\en{W},\overline{S}\}$.

Using the definition of homs in $[\en{D},\en{V/I}]$ we have a pullback
\[ \xymatrix @R2pc @C2pc {
[\en{D},\en{V/I}](\en{W},R(\cb)(B,\overline{S}-)) \ar[r]^{U} \ar[d]_{} & [\cd,\cv](W,\cb(B,S-)\x
ID) \ar[d]^{{(\pi_2)_{*}} } \\
I \ar[r]_{w}& [\cd,\cv](W,ID) 
  } \]
and so, using the universal property of the product $\cb(B,S-)\x ID$,
an isomorphism
$$[\en{D},\en{V/I}](\en{W},R(\cb)(B,\overline{S}-)) \cong [\cd,\cv](W,\cb(B,S-))\x I$$
over $I$, and now the right hand side is naturally isomorphic to
\[\cb(B,\{W,S\})\x I = R(\cb)(B,\{W,S\}),\]
as required.  This proves that $\{W,S\}$ also has the universal property of
$\{\en{W},\overline{S}\}$.  Given this, the corresponding statement about preservation is straightforward.
\end{proof}

\begin{lemma}\label{lem:pres}
If $\cb$ has enough cofibrantly-weighted limits and $U\colon\cb \to
\ca$ preserves them, then $R(\cb)$ has enough cofibrantly-weighted
limits (in the $\cv/I$-enriched sense) and $R(U)\colon R(\cb) \to R(\ca)$ preserves them.
\end{lemma}
\begin{proof}
Let $\en{D}$ be a $\cv/I$-category.  Re-expressing the definition of enough cofibrantly-weighted limits in these terms, we see that there exists a weight $\en{W}=(W,w)$ with $W$ cofibrant and $w\colon W \to ID$ a trivial fibration, and such that $\cb$ has $W$-weighted limits and $U$ preserves them.  By Lemma~\ref{lem:basechange}, this implies that $R(\cb)$ has $\en{W}$-weighted limits and $R(U)$ preserves them.

Now $(ID,1)$ is the terminal $\cv/I$-weight so that the trivial fibration $w\colon W \to ID$ equally specifies a trivial fibration $w\colon\en{W} \to 1$.  Therefore, if we can show that $W$ cofibrant implies $\en{W}$ cofibrant, we will have shown that $R(\cb)$ has enough cofibrantly-weighted limits as a $\cv/I$-category and that $R(U)$ preserves them.

To this end, suppose that $\alpha\colon\en{F} \to \en{G}$ is a trivial
fibration and consider the commutative square in the top half of
the following diagram. 
\begin{equation*}
\xymatrix{
[\en{D},\en{V/I}](\en{W},\en{F}) \ar[d]_{\alpha_*} \ar[rr]^{U_{W,F}} && [\cd,\cv](W,F) \ar[d]^{\alpha_{*}} \\
[\en{D},\en{V/I}](\en{W},\en{G}) \ar[d] \ar[rr]^{U_{W,G}} && [\cd,\cv](W,G)\ar[d]^{g_*} \\
 I \ar[rr]_{w} && [\cd,\cv](W,ID)}
 \end{equation*}
Pasting this with the pullback square \eqref{eq:hom} defining
$[\en{D},\en{V/I}](\en{W},\en{G})$ as in the bottom half of the
diagram yields the pullback square defining $[\en{D},\en{V/I}](\en{W},\en{F})$.  Therefore the upper square above is a pullback in $\cv$. Since $W$ is cofibrant, the right hand $\alpha_*$ is a trivial fibration whence so is its pullback on the left, as required.
\end{proof}

With this in place, we can finally prove the missing proposition, so completing the proof of Theorem~\ref{thm:adjunction}.

\begin{proposition}\label{cor:enough_cofibrant_limits_comma}
If $\cb$ has enough cofibrantly-weighted limits and $U$ preserves them, in the $\cv$-sense, then $A/U$ has enough cofibrantly-weighted limits in the $\cv/I$-sense.
\end{proposition}

\begin{proof} 
By Lemma~\ref{lem:pres} $R(\cb)$ has enough cofibrantly-weighted limits and $R(U)\colon R(\cb) \to R(\ca)$ preserves them.  Therefore, by Proposition~\ref{prop:completeness}, the comma $\cv/I$-category $A/R(U)$ has enough cofibrantly-weighted limits.  By Lemma~\ref{lem:comma2}, $A/R(U)$ is isomorphic to the $\cv/I$-category $P\colon A/U \to \ci$, which consequently has enough cofibrantly-weighted limits too.
\end{proof}

In the case of the ordinary adjoint functor theorem, there is of
course a converse: if a functor has a left adjoint then it satisfies
the solution set condition and preserves any existing limits.
In the case of Theorem~\ref{thm:adjunction}, if there is a left shrink-adjoint, then the solution set
condition does hold (as stated in the theorem), but the preservation
of enough cofibrantly-weighted limits need not, as the following 
example shows.

\begin{example}
   Let $\cv = \Set$ with the split model structure, so that the
   shrinkable morphisms are just the surjections, and we are dealing with 
   classical weak category theory. Let $G$ be a non-trivial group, and
   $\Set^G$ the category of $G$-sets. The forgetful functor $U\colon
   \Set^G\to\Set$ has a left adjoint $F$.
   Let $P\colon\Set^G\to\Set$ be the functor sending a $G$-set to
   its set of orbits. There is a natural transformation $\pi\colon
   U\to P$ whose components are surjective. For each set $X$, the
   composite of the unit $X\to UFX$ and $\pi\colon UFX\to PFX$
   exhibits $F$ as a weak left adjoint to $P$.
   But $P$ does not preserve powers (or products): in particular, it does not preserve the power $G^2$.
\end{example}

This shows that having a left shrink-adjoint does not imply
preservation of limits up to isomorphism, as would be needed for a
genuine converse. A more reasonable request would be preservation in
some homotopical sense. It turns out that in good cases this does
hold, as in the following result, which can be seen as a sort of
partial converse to our adjoint functor theorem. 

\begin{theorem}\label{thm:preservation}
  Let \cb and \ca be locally fibrant \cv-categories, and let
  $U\colon\cb\to\ca$ have a left shrink-adjoint. Let $Q\colon\cd\to\cv$ be a cofibrant weight for which $[\cd,\cv](Q,-)\colon [\cd,\cv]\to\cv$ sends levelwise shrinkable morphisms between levelwise fibrant objects to weak equivalences. Then $U$ preserves homotopy $Q$-weighted limits. 
\end{theorem}

\proof
The induced maps $\cb(A',B)\to\ca(A,UB)$ are shrinkable morphisms between fibrant objects and so are weak equivalences (between fibrant objects).
Suppose that $\phi\colon Q\to\cb(L,S)$ exhibits $L$ as the homotopy weighted limit $\{Q, S\}_h$, meaning that for each $B \in \cb$ the induced morphism
$\cb(B,L) \to [\cd,\cv](Q,\cb(B,S))$ is a weak equivalence in $\cv$.  Then given $A\in\ca$ there is a commutative diagram
\[ \xymatrix @R2pc @C2pc {
      \cb(A',L) \ar[r] \ar[d] & [\cd,\cv](Q,\cb(A',S)) \ar[d] \\
      \ca(A,UL) \ar[r] & [\cd,\cv](Q,\ca(A,US))
    } \]
in which the upper horizontal is the weak equivalence expressing the
universal property of $L=\{Q, S\}_h$, the left vertical is a shrinkable
morphism (and so a weak equivalence) expressing the universal property of
$A'$, and the right vertical is given by $[\cd,\cv](Q,-)$ applied to
a levelwise shrinkable morphism between levelwise fibrant objects, and so is a weak equivalence. Then the bottom map is also a weak equivalence, and so $U$ preserves the homotopy limit.
\endproof

\begin{example}
In \cite{LackRosicky2012}, $Q$-weighted limits were said to be {\em \ce-stable} if $\ce$, seen as a full subcategory of $\cv^\two$, is closed under $Q$-weighted limits; in other words, if $[\cd,\cv](Q,-)\colon[\cd,\cv]\to\cv$ sends a $\cv$-natural transformation whose components are in $\ce$ to a morphism in $\ce$. If moreover the \dsdr\ morphisms are weak equivalences then the hypothesis in the theorem holds, and so $U$ will preserve homotopy $Q$-weighted limits. This is the case for all cofibrant $Q$ in the case of Examples~\ref{ex:trivial}, \ref{ex:split}, and \ref{ex:Cat}, but not in Example~\ref{ex:Joyal}.
\end{example}

\begin{example}
If the enriched projective model structure on $[\cd,\cv]$ exists, as is the case for Example~\ref{ex:Joyal}, then once again the hypothesis holds for all cofibrant $Q$. For a levelwise shrinkable morphism between levelwise fibrant objects is a weak equivalence between fibrant objects in the projective model structure, and so sent by $[\cd,\cv](Q,-)$ to a weak equivalence (between fibrant objects).
\end{example}

\section{Accessibility and shrink-colimits}\label{sect:cocompleteness}

In ordinary category theory an accessible category which is complete
is also cocomplete --- this provides one simple way to see that
algebraic categories, which are obviously complete, are also
cocomplete.

The present section is devoted to generalising this to our setting,
which we do in Theorem~\ref{thm:colimits}.
In Section~\ref{sect:applications} we shall use
Theorem~\ref{thm:colimits}
to deduce homotopical cocompleteness of various enriched categories of (higher) categorical structures.

\subsection{Shrink-colimits}

The paper \cite{LackRosicky2012} introduced the notion of $\ce$-weak
colimit for any class of morphisms $\ce$. We shall use this in
our current setting where $\ce$ consists of the \dsdr\
morphisms, and call the resulting notion a {\em
  shrink-colimit}. 

\begin{definition}
Let $W\colon \cd\op \to \cv$ be a weight, and consider a diagram
$S\colon \cd \to \ca$.  A $\cv$-natural transformation
\[ \xymatrix @R2pc @C2pc {
W \ar[r]^-{\eta} & \ca(S-,C) 
  } \]
in $[\cd\op,\cv]$ exhibits $C$ as the shrink-colimit of $S$ weighted by $W$ if the induced map
\begin{equation}\label{eq:weakcolimit}
\xymatrix{
\ca(C,A) \ar[rr]^-{\eta^{*} \circ \ca(S-,1)} &&
[\cd\op,\cv](W,\ca(S-,A)) 
}
\end{equation}
 is shrinkable for all $A\in\ca$.
\end{definition}
This is equally the assertion that $\eta$ exhibits $C$ as a
shrink-reflection of $W$ along the \cv-functor
$\ca(S-,1)\colon\ca \to [\cd\op,\cv]$ sending $A$ to $\ca(S-,A)$.

\begin{example}\label{ex:colimits-trivial}
 For $\cv$ with the trivial model structure, shrink-colimits are weighted colimits in the usual sense.
 \end{example}
 
\begin{example}\label{ex:colimits-split}
 For $\cv$ with the split model structure, $\eta\colon W \to
  \ca(S-,C)$ exhibits $C$ as the shrink-colimit just when the
  induced map $\ca(C,A) \to [\cd\op,\cv](W,\ca(S-,A))$ is a split
  epimorphism in $\cv$.  This implies in particular that given $f\colon W \to
  \ca(S-,A)$ there exists $f'\colon C \to A$ such that the triangle
\begin{equation}\label{eq:fact}
\xymatrix{
W \ar[d]_{\eta} \ar[r]^{f} & \ca(S-,A) \\
\ca(S-,C) \ar[ur]_{\ca(-,f')} 
}
\end{equation}
commutes. Indeed, when $\cv = \Set$, it amounts to precisely this
condition, and then if $W$ is the terminal weight, the shrink-colimit reduces to the ordinary weak (conical) colimit of $S$.
\end{example}

\begin{example}\label{ex:colimits-2cats}
 In the $\Cat$ case, the map $$\ca(C,A) \to
  [\cd\op,\Cat](W,\ca(S-,A))$$ is required to be a surjective equivalence.  So
  given $f\colon W \to \ca(S-,A)$ we have a factorisation as in
  \eqref{eq:fact} above; and furthermore, given $\alpha\colon f
  \Rightarrow g \in [\cd\op,\Cat](W,\ca(S-,A))$  there exists a
  unique $\alpha'\colon f' \Rightarrow g'$ such that $\ca(S-,\alpha') \circ \eta = \alpha$.
For instance, the shrink-coequalizer 
\begin{equation*}
\xymatrix{
X \ar@<1ex>[r]^{f} \ar@<-1ex>[r]_{g} & Y \ar[r]^{e} & C
}
\end{equation*}
satisfies $e \circ f = e \circ g$ and has the following properties:
\begin{enumerate}
\item if $h\colon Y \to D$ satisfies $h\circ f=h\circ g$, there
  \emph{exists} $h^{\prime}\colon C \to D$ such that $h^{\prime} \circ e = h$;
\item if moreover $k\colon Y \to D$ satisfies $k\circ f=k\circ g$, and
  $\theta\colon h \Rightarrow k$ satisfies $\theta \circ f = \theta
  \circ g$, there exists a unique $\theta^{\prime}\colon h^{\prime} \Rightarrow k^{\prime}$ such that $\theta^{\prime} \circ e = \theta$.
\end{enumerate}

Let us compare shrink-colimits with the better known notion of
bicolimits --- given $W$ and $S$ as before, the $W$-weighted bicolimit
$C=W*_b S$ \cite{Kelly-limits} is defined by a
\emph{pseudonatural transformation} $\eta\colon W \to\ca(S-,C)$ such
that the induced map $$\ca(C,A) \to \Ps(\cd\op,\Cat)(W,\ca(S-,A))$$
is an equivalence of categories, where $\Ps(\cd\op,\Cat)$ denotes
the $2$-category of $2$-functors, pseudonatural transformations and
modifications from $\cd\op$ to $\Cat$.  
To make the comparison,
recall from \cite[ Remark~3.15]{Blackwell1989Two} that the identity on
objects inclusion of $[\cd\op,\Cat]$ in $\Ps(\cd\op,\Cat)$ has a left
adjoint $(-)^{\prime}$ called the \emph{pseudomorphism classifier},
with counit $q_W\colon W^{\prime} \to W$.  The morphism $q_W \colon
W^{\prime} \to W$ is in fact a cofibrant replacement of the weight $W$
in the projective model structure on $[\cd\op,\Cat]$: see
\cite[ Sections~5.9 and~6]{Lack2007Homotopy-theoretic}.
Given the isomorphism
\[ [\cd\op,\Cat](W^{\prime},\ca(S-,A)) \cong
  \Ps(\cd\op,\Cat)(W,\ca(S-,A)), \]
the bicolimit of $S$ weighted by $W$
equally amounts to a $2$-natural transformation $\eta\colon W^{\prime}
\to \ca(S-,C)$ for which the induced map
\[\ca(C,A) \to [\cd\op,\Cat](W^{\prime},\ca(S-,A))\]
is an \emph{equivalence of
  categories} for all $A$.  In particular, the shrink-colimit of $S$
weighted by $W^{\prime}$ is a $W$-weighted bicolimit of $S$,
but satisfies the stronger condition that genuine factorisations, as
in \eqref{eq:fact}, exist.  Thus any  $2$-category admitting shrink-colimits also admits bicolimits (with a stronger universal
property), but also admits some shrink-colimits, such as shrink-coequalizers, whose defining weights are not cofibrant, and so do not correspond to any bicolimit. 
\end{example}

\begin{example}\label{ex:colimits-cosmoi}
In the $\SSet$-case, the map $$\ca(C,A) \to [\cd\op,\SSet](W,\ca(S-,A))$$ is a \dsdr\ morphism and so a weak equivalence (with respect to the Joyal model structure) with a section.  

In the case that $W$ is flexible and $\ca$ is locally fibrant --- as,
for instance, if $\ca$ is an $\infty$-cosmos --- we can say rather more.  For then, since $\ca(S-,A)$ is pointwise fibrant in $[\cd\op,\SSet]$ and $W$ is cofibrant, the hom $[\cd\op,\SSet](W,\ca(S-,A))$ is also fibrant: it is a quasicategory.  It follows, by Example~\ref{ex:shrinking}, that the \dsdr\ morphism $\ca(C,A) \to [\cd\op,\SSet](W,\ca(S-,A))$ is a \emph{surjective equivalence of quasicategories}.

For $\ca$ an $\infty$-cosmos and $W$ a flexible weight, Riehl and
Verity \cite{Riehl2018On} define the \emph{flexibly-weighted homotopy
  colimit} of $S$ weighted by $W$ as an object $C$ together with a
morphism $W \to \ca(S-,C)$ for which the induced map $\ca(C,A) \to
[\cd\op,\SSet](W,\ca(S-,A))$ is an equivalence of quasicategories for
all $A$.  In particular, if $\ca$ admits shrink-colimits it admits flexibly-weighted homotopy colimits with the stronger property that the comparison equivalence of quasicategories is in fact a surjective equivalence.
\end{example}

\subsection{Accessible $\cv$-categories with enough cofibrantly-weighted limits have shrink-colimits}

In the present section we suppose that (the unenriched category) $\cv_0$ is locally
presentable. It follows by \cite[Proposition~2.4]{vcat} that there is a regular cardinal
$\lambda_0$ such that $\cv_0$ is locally $\lambda_0$-presentable, the unit
object $I$ is $\lambda_0$-presentable, and the tensor product of two
$\lambda_0$-presentable objects is $\lambda_0$-presentable. Moreover, the
corresponding statements will remain true for any regular $\lambda \geq \lambda_0$.

\begin{notation}\label{notation-lambda}
  We let $\lambda_0$ denote a fixed regular cardinal as in the
  previous paragraph. Whenever we consider
  $\lambda$-presentability or $\lambda$-accessibility in the $\cv$-enriched context, 
    we shall always suppose that $\lambda$ is a regular
  cardinal and $\lambda\ge\lambda_0$. 
\end{notation}

If $\ch$ is a small category we can speak of conical $\ch$-shaped colimits in
$\cc$: these are $\ch$-shaped colimits in $\cc_0$ which are (required to be) preserved by
$\cc(-,C)_0\colon \cc_0\to \cv\op_0$ for each $C\in\cc$.
In particular we can speak of $\lambda$-filtered colimits in $\cc$,
corresponding to $\ch$-shaped colimits for all $\lambda$-filtered
categories $\ch$.

An object $A$ in $\cc$ is said to be \emph{$\lambda$-presentable} if
$\cc(A,-)\colon \cc \to \cv$ preserves $\lambda$-filtered colimits.
We say that $\cc$ is $\lambda$-\emph{accessible}\footnote{A 
 different notion of enriched accessibility was given in
 \cite{BorceuxQuinteiro}, but see \cite{LackTendas-Flat} for an
 analysis of the relationship between the two notions, including the
 fact that they agree in the crucial examples $\cv=\Cat$ and $\cv=\SSet$.}
if it admits
$\lambda$-filtered colimits and a set $\cg$ of $\lambda$-presentable
objects such that each object of $\cc$ is a $\lambda$-filtered colimit
of objects in $\cg$. It follows, arguing as usual, that the full
subcategory of $\lambda$-presentable objects in $\cc$ is essentially
small and we denote by $J\colon \cc_{\lambda} \to \cc$ a small
skeletal full subcategory of $\lambda$-presentables.
It follows from \cite[Theorem~5.19]{Kelly1982Basic} that $J$ is
\emph{dense}, and from \cite[Theorem 5.29]{Kelly1982Basic} that $J$
then exhibits $\cc$ as the free completion of $\cc_{\lambda}$ under
$\lambda$-filtered colimits.
(Indeed, the $\lambda$-accessible $\cv$-categories are equally the free completions of small $\cv$-categories under $\lambda$-filtered colimits.)

\begin{notation}\label{notation:nerve} 
  Let $D\colon\ca\to\cc$ be a \cv-functor with small domain. We write
  $N_D$ for the induced \cv-functor $\cc\to[\ca\op,\cv]$ sending an
  object $C\in\cc$ to the presheaf $\cc(D-,C)$ which in turn sends $A\in\ca$
  to $\cc(DA,C)$. Other authors have written $\cc(D,1)$ or $\widetilde{D}$
  for $N_D$; our notation is designed to remind that this is a
  generalised (N)erve.
\end{notation}

Let $\cc$ be a $\lambda$-accessible $\cv$-category, and $J\colon
\cc_{\lambda} \to \cc$ the inclusion of the full sub-$\cv$-category of $\lambda$-presentable objects.  Then $J$ is dense, so that the associated functor $N_J\colon \cc \to [\cc_{\lambda}\op,\cv]$ is fully faithful.

\begin{proposition}\label{prop:reflective}
If $\cc$ is $\lambda$-accessible, with powers and enough cofibrantly-weighted limits, then $N_J\colon \cc \to [\cc_{\lambda}\op,\cv]$ admits a left shrink-adjoint.
\end{proposition}
\begin{proof}
  The composite of $N_J$ with the evaluation functor $\ev_{X}$ at a
$\lambda$-presentable object $X$ is the representable $\cc(JX,-)$.
Each such $\cv$-functor preserves $\lambda$-filtered colimits; so, since the evaluation functors are jointly
conservative, $N_J$ also preserves $\lambda$-filtered colimits. 
Now $[\cc_{\lambda}\op,\cv]$ is locally $\lambda$-presentable as a
$\cv$-category by \cite[Examples 3.4]{Kelly1982Structures}.  In particular, since $(N_J)_0\colon \cc_0 \to [\cc_{\lambda}\op,\cv]_0$ is a $\lambda$-filtered colimit preserving functor between $\lambda$-accessible categories, it satisfies the solution set condition.  The result now follows from our adjoint functor theorem, Theorem~\ref{thm:adjunction}.
\end{proof}

\begin{theorem}\label{thm:colimits}
Let $\cc$ be an accessible $\cv$-category with powers and enough cofibrantly-weighted limits.  Then $\cc$ admits all shrink-colimits.
\end{theorem}
\begin{proof}
Given Proposition~\ref{prop:reflective}, this follows directly from
Proposition 4.3 of \cite{LackRosicky2012}, the argument of which we
repeat here for convenience.  Let $D\colon \ca \to \cc$ with $\ca$
small.  We must show that $N_D\colon \cc \to [\ca\op,\cv]$ has a
left shrink-adjoint; then the value of this adjoint at
$W\in[\ca\op,\cv]$ will give the shrink-colimit of $D$ weighted by
$W$.  First suppose that $\cc$ is $\lambda$-accessible and consider the dense inclusion $J\colon \cc_{\lambda} \to \cc$.  We then have the following diagram.
\begin{equation*}
\xymatrix{
\ca \ar[r]^{D} & \cc \ar[drr]_{N_D} \ar[rr]^{N_{J}} && [\cc_{\lambda}\op,\cv] \ar[d]^{N_{(N_{J}D)} = [\cc_{\lambda}\op,\cv](N_JD-,1)} \\
&&& [\ca\op,\cv]
}
\end{equation*}
Since $N_J$ is fully faithful, there are natural isomorphisms
\[ N_{N_JD} N_J C \cong [\cc_{\lambda}\op,\cv](N_J D-,N_JC) \cong \cc(D-,C)
  \cong N_D C \]
and so the triangle commutes up to isomorphism. Thus it will suffice
to show that $N_J$ and $N_{N_JD}$ have left
shrink-adjoints. 
The first of these does so by Proposition~\ref{prop:reflective}, while
the second has a genuine left
adjoint, namely the weighted colimit functor ${N_J}D \star -$. 
\end{proof}

We follow \cite{LackRosicky2012} in saying that $W$-weighted limits
are \emph{$\ce$-stable} if the enriched full subcategory $\ce
\hookrightarrow \cv^{\atwo}$ is closed under $W$-weighted limits.
Similarly, we say that {\em enough cofibrantly-weighted limits are
  $\ce$-stable} if \ce has, and the inclusion $\ce\to\cv^{\atwo}$
preserves, enough cofibrantly-weighted limits. In that case we can use results
from \cite{LackRosicky2012} to prove a converse to the previous
theorem.  

\begin{theorem}\label{thm:colimits2}
Suppose that the \dsdr\ morphisms in $\cv$ are closed in $\cv^{\atwo}$ under enough cofibrantly-weighted limits, and let $\cc$ be an accessible $\cv$-category.  Then the following are equivalent.
\begin{enumerate}[(i)]
\item $\cc$ admits $\ce$-stable limits.
\item $\cc$ admits powers and enough cofibrantly-weighted limits.
\item $\cc$ admits shrink-colimits.
\end{enumerate}
\end{theorem}
\begin{proof}
Since \dsdr\ morphisms are always closed under powers, powers are
$\ce$-stable.  Since also enough cofibrantly-weighted limits are
$\ce$-stable by assumption, $(i) \implies (ii)$.  The
implication $(ii) \implies (iii)$ holds by
Theorem~\ref{thm:colimits}. For $(iii)\implies (i)$, suppose that $\cc$ is $\lambda$-accessible and consider
the dense inclusion $J\colon \cc_{\lambda} \to \cc$.  Then $N_J\colon \cc
\hookrightarrow [\cc\op_{\lambda},\cv]$ is fully faithful and has
 a left shrink-adjoint sending $W$ to the shrink-colimit of $J$
weighted by $W$.  Furthermore, since each morphism of $\ce$ is a split
epimorphism, the representable $\cv(I,-)\colon\cv_0 \to \Set$ sends
morphisms of $\ce$ to surjections: in the language of
\cite{LackRosicky2012} this says that $I$ is \emph{$\ce$-projective}.
Using Propositions~6.1 and~2.6 of \cite{LackRosicky2012} we deduce that $\cc$ is closed in $[\cc\op_{\lambda},\cv]$ under $\ce$-stable limits, as required.
\end{proof}

\subsection{Recognising accessible enriched categories}

Examples of accessible $\cv$-categories $\cc$ are much easier to
identify in the case that $\cc$ admits powers by a suitable strong
generator $\cg$ of $\cv_0$. As observed in
\cite[Section~3.8]{Kelly1982Basic}, if the $\cv$-category $\cc$ has
powers, the difference between conical colimits in $\cc$ and conical
colimits in the underlying category $\cc_0$ disappears. But since the
map $\cc(\colim_i D_i,C)\to \lim_i \cc(D_i,C)$ expressing the
universal property of a conical colimit will be invertible if and
only if the induced $\cv_0(G,\cc(\colim_iD_i,C))\to
\cv_0(G,\lim_i\cc(D_i,C))$ is so for each $G\in\cg$,
it suffices for $\cc$ to have powers by objects in $\cg$. 
Building on this well-known argument,
the following proposition enables us to recognise accessible
$\cv$-categories by looking at their underlying categories.

Recall that $\lambda_0$ satisfies the standing assumptions of
Notation~\ref{notation-lambda}.

\begin{proposition}\label{prop:equivalent}
  Let $\cv_0$ be locally presentable, let $\cg$ be a strong generator
  of $\cv_0$, and let $\lambda\ge\lambda_0$ be a regular cardinal such
  that each $G\in\cg$ is $\lambda$-presentable in $\cv_0$.  Then for
  any $\cv$-category $\cc$ with powers by objects in $\cg$, the
  following are equivalent:
  \begin{enumerate}
  \item $\cc$ is $\lambda$-accessible as a $\cv$-category;
  \item $\cc_0$ is $\lambda$-accessible and
    $\cc(C,-)_0\colon \cc_0\to\cv_0$ is a $\lambda$-accessible
    functor, for each $C$ in some strong generator of $\cc_0$;
  \item $\cc_0$ is $\lambda$-accessible and
    $(G\pitchfork -)_0\colon \cc_0\to\cc_0$ is a $\lambda$-accessible
    functor, for each $G\in\cg$.
  \end{enumerate}
  Moreover, for such $\lambda$ and $\cc$, an object $C\in\cc$ is
  $\lambda$-presentable if and only if it is $\lambda$-presentable in
  $\cc_0$.
\end{proposition}
\begin{proof}
  Because of the powers, $\cc$ has $\lambda$-filtered colimits if and
  only if $\cc_0$ does so, and a class $\ch$ of objects of $\cc$
  generates $\cc$ under $\lambda$-filtered colimits if and only if it
  generates $\cc_0$ under $\lambda$-filtered colimits.  So the only
  possible difference between $\lambda$-accessibility of $\cc$ and
  $\lambda$-accessibility of $\cc_0$ lies in the possible difference
  between $\lambda$-presentability in $\cc$ and
  $\lambda$-presentability in $\cc_0$.

  Since $I$ is $\lambda$-presentable in $\cv_0$, if $C$ is
  $\lambda$-presentable in $\cc$ then the composite
  \[ \xymatrix @R2pc @C3pc { \cc_0 \ar[r]^-{\cc(C,-)_0} & \cv_0
      \ar[r]^-{\cv_0(I,-)} & \Set } \] preserves $\lambda$-filtered
  colimits, but this composite is $\cc_0(C,-)$, and so $C$ is
  $\lambda$-presentable in $\cc_0$. Thus
  $(\cc_\lambda)_0\subseteq(\cc_0)_\lambda$; while if $\cc$ is
  $\lambda$-accessible as a $\cv$-category, then $\cc_0$ is an
  accessible ordinary category, and (1) implies (2).

  For an arbitrary $\cg$-powered $\cv$-category $\cc$, however, a
  $\lambda$-presentable object of $\cc_0$ might fail to be
  $\lambda$-presentable in $\cc$.

  For the remaining implications, consider the diagram
  \begin{equation}
    \label{eq:G-powers}
    \xymatrix  @R2pc @C2pc{
      \cc_0 \ar[rr]^{(G\pitchfork-)_0} \ar[d]_{\cc(C,-)_0} && \cc_0
      \ar[d]^{\cc_0(C,-)} \\
      \cv_0 \ar[rr]_{\cv_0(G,-)} && \Set 
    }  
  \end{equation}
  in which $C\in\cc$, $G\in\cg$, and so the lower horizontal preserves
  $\lambda$-filtered colimits.

  Suppose that (2) holds, so that there is a strong generator $\ch$
  for $\cc_0$ consisting of objects which are $\lambda$-presentable in
  \cc (and so also in $\cc_0$). As $C$ ranges through $\ch$, the lower
  composite $\cv_0(G,\cc(C,-)_0)$ in \eqref{eq:G-powers} preserves
  $\lambda$-filtered colimits, while the $\cc_0(C,-)$ preserve and
  jointly reflect them. Thus $(G\pitchfork-)_0$ preserves them and (3)
  holds.

  Now suppose that (3) holds. If $C$ is $\lambda$-presentable in
  $\cc_0$, then the upper composite $\cc_0(C,G\pitchfork-)$
  in~\eqref{eq:G-powers} preserves $\lambda$-filtered colimits, while
  the $\cv_0(G,-)$ preserve and jointly reflect them, thus
  $\cc(C,-)_0$ also preserves them, and $C$ is $\lambda$-presentable
  in $\cc$. Thus $(\cc_0)_\lambda\subseteq (\cc_\lambda)_0$ and (1)
  follows.
\end{proof}

In practice, it is often convenient to lift enriched accessibility
from one (typically locally presentable) $\cv$-category to another
using the following.

\begin{corollary}\label{cor:recognition}
  Let $\ca$ and $\cc$ be $\cv$-categories with powers by objects in
  $\cg$ and let $U\colon \ca \to \cc$ be a \emph{conservative}
  $\cv$-functor preserving them.  Suppose that $\cc$ is an accessible
  $\cv$-category, $\ca_0$ is an accessible category, and
  $U_0\colon\ca_0 \to \cc_0$ is an accessible functor.  Then in fact
  $\ca$ is accessible as a $\cv$-category.
\end{corollary}
\begin{proof}
  Consider the following commutative diagram
  \[ \xymatrix @R2pc @C2pc {
      \ca_0 \ar[rr]^-{(G\pitchfork-)_0} \ar[d]_{U_0} && \ca_0 \ar[d]^{U_0} \\
      \cc_0 \ar[rr]_{(G\pitchfork-)_0} && \cc_0 } \]
  where $G\in\cg$.

  The categories $\ca_0$ and $\cc_0$ are both accessible,
$U_0$ is accessible by assumption, and the lower horizontal by
Proposition~\ref{prop:equivalent} and the fact that \cc is an
accessible $\cv$-category.  By the uniformization theorem \cite[Theorem~2.19]{Adamek1994Locally} for
accessible categories, we can choose $\lambda\ge\lambda_0$ such that each
of the aforementioned categories and functors is
$\lambda$-accessible.
Since $U_0$ is conservative, the upper horizontal is also $\lambda$-accessible.
Therefore $\ca$ is a $\lambda$-accessible $\cv$-category by
Proposition~\ref{prop:equivalent} once again.
\end{proof}

\section{Examples and applications}\label{sect:applications}
 
To summarise, we have two main results: an adjoint functor theorem,
and a shrink-cocompleteness theorem for accessible $\cv$-categories
with sufficient limits. 
In this final section we illustrate the scope of these results by describing what they capture in our various settings, with a particular emphasis on $2$-categories and $\infty$-cosmoi.

\subsection{The classical case}
In the case of $\Set$ with the trivial model structure, a category
$\ca$ has powers and enough cofibrantly-weighted limits just when it
is complete. The \dsdr\
morphisms in $\Set$ are the bijections, and these are of course closed
in $\Set^\two$ under (all) limits. Therefore our weak adjoint functor
theorem specialises exactly to the general adjoint functor theorem of
Freyd \cite{CWM}. Theorem~\ref{thm:colimits2} becomes the well-known
result that an accessible category is complete if and only if it is 
cocomplete: see for example
\cite[Corollary~2.47]{Adamek1994Locally}. Since preservation of homotopy limits is just preservation of limits in this case, Theorem~\ref{thm:preservation} is just the elementary fact that right adjoints preserve limits. 

For general $\cv$ equipped with the trivial model structure, a
$\cv$-category $\ca$ has powers and enough cofibrantly-weighted limits
just when it is complete, now in the sense of weighted limits, and the
\dsdr\  morphisms are again the isomorphisms.
Therefore our adjoint functor theorem specialises exactly to the adjoint functor
theorem for enriched categories. Theorem~\ref{thm:preservation} says that right adjoints preserve limits, and Theorem~\ref{thm:colimits2} yields the well-known result that an
accessible $\cv$-category is complete if and only if it is cocomplete.

\subsection{Ordinary weakness}
In the case of $\Set$ with the split model structure we know from
Example ~\ref{ex:split-enough} that for a category $\ca$ to have enough cofibrantly-weighted limits it suffices that it admit products.  Furthermore the \dsdr\
morphisms are the surjections, and these are closed under
products.
Since powers are products in the $\Set$-enriched setting our adjoint functor theorem thus yields
the weak adjoint functor theorem of Kainen \cite{Kainen1971Weak}.
Theorem~\ref{thm:colimits2} yields the well known result that an
accessible category has products if and only if it has weak colimits: see for example
\cite[Theorem~4.11]{Adamek1994Locally}.

For general $\cv$ equipped with the split model structure, our results
appear to be completely new.
In this setting a $\cv$-category $\ca$ has powers and enough cofibrantly-weighted limits
provided that it has powers and products, and the \dsdr\  morphisms
are the split epimorphisms, which are of course closed under
powers and products. Using this, our main results (in slightly weakened form) are:
\begin{itemize}
\item  Let $\cb$ be a \cv-category with products and powers, and let
  $U\colon\cb\to\ca$ be a \cv-functor that preserves them. Then $U$
  has a (split epi)-weak left adjoint if and only if it satisfies the solution
  set condition.
\item An accessible $\cv$-category has products and powers if and only
  if  it has (split epi)-weak colimits.
\end{itemize}

\subsection{$2$-category theory}
One of our guiding topics is that of $2$-category theory, understood
as $\Cat$-enriched category theory, and  where $\Cat$ is equipped with the canonical model structure.  

As recalled in Example~\ref{ex:enough-Cat}, in this case the
cofibrantly-weighted limits are precisely the flexible limits of
\cite{Bird1989Flexible}, and any 2-category with flexible limits has
powers.
The \dsdr\  morphisms are the retract
equivalences, and these are closed in $\Cat^\two$ under
cofibrantly-weighted limits as observed, for example, in
\cite[Section~9]{LackRosicky2012}.

Our primary applications will be in the accessible setting which we
describe now.
Since each accessible functor between accessible categories satisfies
the solution set condition,
our adjoint functor theorem in this setting gives the first part of the
following,
the second part of which is the instantiation of (part of) Theorem~\ref{thm:colimits2}.

\begin{itemize}
\item  Let $U\colon\cb\to\ca$ be an accessible $2$-functor between accessible $2$-categories.  If $\cb$ has flexible limits and $U$ preserves them then it has a left shrink-adjoint.
\item  An accessible $2$-category has flexible limits if and only if it has shrink-colimits.
\end{itemize}
The first result is new; the second is part of
Theorem 9.4 of \cite{LackRosicky2012}.  The
utility of these results lies in the fact that many, though not all,
$2$-categories of pseudomorphisms are in fact accessible with flexible
limits.  For instance, the $2$-category of monoidal categories and
strong monoidal functors is accessible (although, as recalled in Section~\ref{sect:initial}, the $2$-category of
\emph{strict} monoidal categories and strong monoidal functors is
not \cite[Section~6.2]{Bourke2019Accessible}).  The difference
between the two cases is that the definition of strict monoidal
category involves equations between objects whereas that of monoidal
category does not --- in terms of the associated $2$-monads this
corresponds to the fact that the $2$-monad for strict monoidal
categories is not cofibrant (=flexible) whereas that for monoidal
categories is. 
One result generalising this is Corollary~7.3 of \cite{Bourke2019Accessible}, which asserts that
if $T$ is a finitary flexible $2$-monad on $\Cat$ then the
$2$-category $\TAlg_p$ of strict algebras and pseudomorphisms is
accessible with flexible limits and filtered colimits, and these are
preserved by the forgetful functor to $\Cat$.  It now follows from our
theorem that any such $2$-category admits shrink-colimits, and in
particular bicolimits, and also that if $f\colon S \to T$ is a
morphism of such $2$-monads, then the induced map $\TAlg_p \to
\SAlg_p$ has a left shrink-adjoint, and so a left biadjoint.
These results concerning bicolimits and biadjoints are not new, being special cases of the results of Section 5 of \cite{Blackwell1989Two}.  

However, as shown in \cite[Section~6.5]{Bourke2019Accessible}, many structures
beyond the scope of two-dimensional monad theory are also accessible.
For instance, the $2$-category of small regular categories and regular
functors is accessible with flexible limits and filtered colimits preserved by the forgetful $2$-functor to $\Cat$.
Similar results hold for Barr-exact categories, coherent categories,
distributive categories, and so on.
It follows from the above theorems that the $2$-categories of such
structures admit shrink-colimits, and so bicolimits,
and this is the simplest proof of bicocompleteness in these examples
that we know.
Again, for such structures the adjoint functor theorem provides
an easy technique for constructing left shrink-adjoints, 
and so biadjoints, to forgetful 2-functors between them.

Theorem~\ref{thm:preservation} asserts that a $2$-functor with a left shrink-adjoint preserves bilimits --- since such $2$-functors are right biadjoints, this is a special case of the well known fact that right biadjoints preserve bilimits.

\subsection{Simplicial enrichment and \SSet-accessible $\infty$-cosmoi}
Our second main motivation concerns the $\infty$-cosmoi of Riehl and
Verity \cite{Riehl2019Elements}.  Here the base for enrichment is
$\SSet$ equipped with the Joyal model structure.  In this setting, for
a simplicially enriched category $\cb$ to admit powers and enough cofibrantly-weighted limits it suffices, as observed in Example~\ref{ex:enough-SSet}, that it admit flexible limits in the sense of
\cite{Riehl2019Elements}.  As observed in Example~\ref{ex:shrinking},
every \dsdr\  morphism is a weak equivalence which has a section.

In this example it can be shown that the \dsdr\ morphisms are {\em not}
closed in $\SSet^\two$ under cofibrantly-weighted limits, so we have
to content ourselves with Theorem~\ref{thm:colimits} rather than
Theorem~\ref{thm:colimits2}. On the other hand, enriched projective
model structures do exist, so Theorem~\ref{thm:preservation} does
imply that $\cv$-functors with left shrink-adjoints preserve cofibrantly-weighted homotopy limits.

 Specialising our main results, we then obtain:
\begin{itemize}
\item  Let \cb be a $\SSet$-category with flexible limits, and let
  $U\colon\cb\to\ca$ be a \cv-functor which preserves them. Then $U$
  has a left shrink-adjoint provided that it satisfies the solution
  set condition.
\item If $\ca$ is an accessible $\SSet$-category with flexible limits then it has shrink-colimits.
\end{itemize}

We wish to examine these results further in the setting of
$\infty$-cosmoi.  By definition, an $\infty$-cosmos is a simplicially
enriched category $\ca$ equipped with a class of morphisms called
isofibrations satisfying a number of axioms --- see Chapter 1 of \cite{Riehl2019Elements}.  For the purposes of the present paper, the reader need only know
that the homs of an $\infty$-cosmos are quasicategories and that it admits flexible limits, in the sense of Examples~\ref{ex:enough-SSet}.

The idea is that the
objects of an $\infty$-cosmos are $\infty$-categories, broadly
interpreted, and that the axioms for an $\infty$-cosmos provide what
is needed to define and work with key structures --- such as limits
and adjoints --- that arise in the $\infty$-categorical world.
For instance one can define what it means for an object of an
$\infty$-cosmos to have limits.  This is model-independent,
in the sense that one recovers the usual notion of quasicategory with limits or complete Segal space with limits on choosing the appropriate $\infty$-cosmos.

The natural structure preserving morphisms between $\infty$-cosmoi are
called {\em cosmological functors}, and cosmological functors preserve flexible limits.  
By a {\em \SSet-accessible $\infty$-cosmos}, we mean an $\infty$-cosmos
which is accessible as a simplicial category\begin{footnote}{There are
    further conditions that one could ask for, such as the
    accessibility of the class of isofibrations, but we shall not
    consider these here: see
    \cite{BourkeLack-AccInftyCosmoi}.}\end{footnote}, whilst a
cosmological functor will be called accessible if it is so as a simplicial functor.  

Our main insight here is that many of the $\infty$-cosmoi that arise
in practice are \SSet-accessible, as are the cosmological functors between
them.  Thus the following specialisations of our main results are
broadly applicable.  We remind the reader that in this context the
shrinking morphisms are certain surjective equivalences of quasicategories --- see Examples~\ref{ex:shrinking}.

\begin{itemize}
\item  Let $U\colon\cb\to\ca$ be an accessible cosmological functor between \SSet-accessible $\infty$-cosmoi.  Then $U$ has a left shrink-adjoint.
\item If $\ca$ is a \SSet-accessible $\infty$-cosmos then it has 
  shrink-colimits, and in particular flexibly-weighted homotopy colimits (see Example~\ref{ex:colimits-cosmoi}.)
\end{itemize}

The first examples of \SSet-accessible $\infty$-cosmoi come from the following result, the first part of which is due to Riehl and Verity \cite{Riehl2019Elements}.
\begin{proposition}
\begin{enumerate}
\item Let $\ca$ be a model category that is enriched over the Joyal
  model structure on simplicial sets, and in which every fibrant
  object is cofibrant.  Then the simplicial subcategory
  $\ca\fib$ spanned by its fibrant objects is an
  $\infty$-cosmos, and is closed in $\ca$ under flexible limits.
\item If $\ca$ is furthermore a \emph{combinatorial} model category
  then $\ca\fib$ is a \emph{\SSet-accessible} $\infty$-cosmos and the
  inclusion $\ca\fib\hookrightarrow \ca$ is an accessible embedding.
\end{enumerate}
\end{proposition}

\begin{proof}
  Part (1) holds by Proposition E.1.1 of \cite{Riehl2019Elements}.
  In Part (2), $\ca_0$ is locally presentable and \ca is cocomplete as
  a $\cv$-category by assumption.  Since for each $X \in \SSet$ the
  functor $(X\pitchfork-)_0 \colon\ca_0 \to \ca_0$ has left adjoint
  $(X.-)_0$ given by taking copowers (or tensors) it is accessible.
  Therefore $\ca$ is accessible as a $\SSet$-category by
  Proposition~\ref{prop:equivalent}. Since $\ca\fib$ is closed in
  $\ca$ under flexible limits, it is closed under powers.
Since $\ca$ is combinatorial, $(\ca\fib)_0 = \Inj(J) \hookrightarrow
\ca_0$ where $J$ is the set of generating trivial cofibrations in
$\ca$.  Now by Theorem~4.8 of \cite{Adamek1994Locally} $\Inj(J)$ is
accessible and accessibly embedded, so by
Corollary~\ref{cor:recognition} we conclude that $\ca\fib$ is
accessible and therefore a \SSet-accessible $\infty$-cosmos.
\end{proof}

As explained in Appendix~E of \cite{Riehl2019Elements}, one obtains as instances of the above the $\infty$-cosmoi of quasicategories, complete Segal spaces, Segal categories and $\Theta_n$-spaces.  Since all of the defining model structures are combinatorial, it follows furthermore that each of these $\infty$-cosmoi is \SSet-accessible. 

The above $\infty$-cosmoi can be regarded as basic.  In future work \cite{BourkeLack-AccInftyCosmoi}
we shall vastly extend the scope of these examples by establishing
that \emph{accessible $\infty$-cosmoi} ($\SSet$-accessible $\infty$-cosmoi satisfying natural additional
properties) are stable under a variety of constructions, such as the
passage from an $\infty$-cosmos $\ca$ to the $\infty$-cosmos
$\Cart(\ca)$  of cartesian fibrations therein.
For (not necessarily accessible) $\infty$-cosmoi this was done in Chapter~6 of \cite{Riehl2019Elements}, and it remains to add accessibility to the mix.  (The corresponding stability results for the appropriate class of accessible $2$-categories have been established in \cite{Bourke2019Accessible}.)  

In the present paper we content ourselves with describing two examples
of \SSet-accessible $\infty$-cosmoi built from the quasicategories example: the $\infty$-cosmoi of quasicategories with a class of limits, and the $\infty$-cosmos of cartesian fibrations of quasicategories.  In order to construct these $\infty$-cosmoi systematically, we shall make use of the generalised sketch categories $C|\cc$ of Makkai \cite{Makkai1997Sketches1}, which we now recall.

Given a category $\cc$ and object $C \in \cc$ we form the category
$C|\cc$, an object of which consists of an object $A \in \cc$ together
with a subset $A_C \subseteq \cc(C,A)$, which we often refer to as a
{\em marking}.  A morphism $f\colon (A,A_C) \to (B,B_C)$ consists of a
morphism $f\colon A \to B$ in $\cc$ such that if $x \in A_C$ then $f \circ x \in B_C$.  The forgetful functor $C|\cc \to \cc$ functor has a left adjoint equipping $A$ with the empty marking $(A,\varnothing)$.  

The following straightforward result is established in Item (9) of Section (1) of \cite{Makkai1997Sketches1}, except that Makkai works with locally {\em finitely} presentable categories.   In its present form, it is a special case of Proposition 2.2 of \cite{Isaev}.
\begin{proposition}
Let $\cc$ be locally presentable and $C\in\cc$.  Then $C|\cc$ is also locally presentable and the forgetful functor $C|\cc \to \cc$ is accessible.
\end{proposition}
If $\cj$ is a set of morphisms in $C|\cc$ then we may form the category $\Inj(\cj)$ of
$\cj$-injectives in $C|\cc$.  We follow Makkai's terminology in calling such a category $\Inj(\cj)$ a \emph{doctrine} in $\cc$. 

\begin{corollary}\label{cor:doctrine}
Let $\cc$ be locally presentable, with $C\in\cc$ and $\cj$ a set of
morphisms in $C|\cc$.  Then $\Inj(\cj)$ is accessible, and the composite forgetful functor $U\colon \Inj(\cj) \to \cc$ is accessible.
\end{corollary}
\begin{proof}
By the preceding proposition $C|\cc$ is locally presentable whence, by Theorem 4.8 of \cite{Adamek1994Locally}, the full subcategory $\Inj(\cj) \hookrightarrow C|\cc$ is accessible and accessibly embedded.  The composite $\Inj(\cj) \to C|\cc \to \cc$ is accessible since its two components are.
\end{proof}
We now combine this with simplicial enrichment.
\begin{corollary}\label{cor:doctrineEnriched}
Let $\ca$ be an $\infty$-cosmos and $U\colon \ca \to \cc$ a
conservative simplicially enriched functor preserving flexible limits
to a locally presentable simplicially enriched category.  Suppose that $U_0\colon \ca_0\to\cc_0$ has the form $\Inj(\cj)\to
C|\cc_0\to\cc_0$ for some $C\in\cc$ and set $\cj$ of morphisms in
$C|\cc_0$. 
Then $\ca$ is accessible as a simplicial
category, and so is a \SSet-accessible $\infty$-cosmos.
\end{corollary}
\begin{proof}
Since flexible limits include powers, this follows immediately from Corollaries~\ref{cor:recognition} and~\ref{cor:doctrine}.
\end{proof}

In the following two examples, we consider the $\infty$-cosmoi
$\QCat_{D}$ of quasicategories with $D$-limits, and $\Cart(\QCat)$ 
of cartesian fibrations respectively.  

These are shown to be $\infty$-cosmoi in Propositions~6.3.13 and~6.3.14, respectively,
of \cite{Riehl2019Elements}; moreover the forgetful simplicial
functors $\QCat_{D} \to \SSet$ and $\Cart(\QCat) \to \SSet^{\atwo}$
preserve flexible limits and are conservative.  To prove that they are \SSet-accessible
$\infty$-cosmoi, it suffices by Corollary~\ref{cor:doctrineEnriched} to
describe the categories $(\QCat_{D})_0$ and $\Cart(\QCat)_{0}$ as
doctrines, respectively, in $\SSet_0$ and in $(\SSet^{\atwo})_0$.

\subsubsection{Quasicategories with limits of a given class}
To begin with, we treat the simpler case of quasicategories with a
terminal object.  To this end, consider $\Delta^0|\SSet$, whose
objects are pairs $(X,U)$ where $X$ is a simplicial set and $U
\subseteq X_0$ a distinguished subset of marked $0$-simplices, and
whose morphisms are simplicial maps preserving marked $0$-simplices.
We shall describe the category $\QCat_T$ of quasicategories admitting
a terminal object and morphisms preserving them as the small 
injectivity class in $\Delta^0|\SSet$ consisting of those $(X,U)$ for which $X$ is a quasicategory and $U\subseteq X_0$ is the set of terminal objects in $X$.
To begin with, we view the inner horn inclusions
\begin{equation}\label{eq:horns}
\{ \Lambda_{k}[n]\colon \Lambda^n_k \to \Delta^n \mid 0 < k < n\}
\end{equation}
 as morphisms of $\Delta^0|\SSet$ in which no $0$-simplex is marked in their source or target --- then $(X,U)$ is injective with respect to the inner horns just when $X$ is a quasicategory.  Now recall (see Definition 4.1 of \cite{Joyal2002} for the dual case of an initial object), that a $0$-simplex $a$ of a quasicategory $X$ is \emph{terminal} if the solid part of each diagram
\begin{equation}\label{eq:terminal}
\xymatrix{
\partial \Delta^n \ar[d]_{j_{n}} \ar[rr]^{f}  && X\\
\Delta^n \ar@{.>}[urr] 
}
\end{equation}
in which $f\colon \partial \Delta^n \to X$ has value $a$ at the final vertex $n \in \partial \Delta^n$ admits a filler.  Accordingly, if $(X,U)$ is a quasicategory with marked $0$-simplices, then \emph{each marked $0$-simplex is terminal} precisely if the solid part of each diagram 
\begin{equation}\label{eq:boundary}
\xymatrix{
(\partial \Delta^n,\{n\}) \ar[d]_{j_{n}} \ar[rr]^{}  && (X,U)\\
(\Delta^n,\{n\}) \ar@{.>}[urr] 
}
\end{equation}
 in $\Delta^0|\SSet$ admits a filler.
The existence of a terminal object can then be expressed via injectivity of $(X,U)$ with respect to the morphism $$(\varnothing,\varnothing) \to (\Delta^0,\{0\})$$ where the unique $0$-simplex $0$ of $\Delta^0$ is marked.  Combining this morphism with the inner horn inclusions \eqref{eq:horns} and the marked boundary inclusions \eqref{eq:boundary} an injective object $(X,U)$ consists of a quasicategory $X$ together with a non-empty subset of terminal objects in $X$.  If we stopped here, the evident full inclusion $\QCat_T \to \Delta^0|\SSet$ would not however be essentially surjective on objects; for this we require our injectives $(X,U)$ to have the property that $U$ consists of \emph{all} terminal objects in $X$.  Since each terminal object is equivalent to any other, it therefore suffices to add the repleteness condition that each object equivalent to one in $U$ belongs to $U$.  To this end, we consider the nerve $J$ of the free isomorphism which has two $0$-simplices $0$ and $1$.  The repleteness condition is captured by the morphism
 \begin{equation*}
 \xymatrix{
 (J,\{0\}) \ar[rr]^{\id} &&  (J,\{0,1\})
 }
 \end{equation*}
 whose underlying simplicial map is the identity.
 
It is straightforward to extend this example to quasicategories with limits of shape $D$ for $D$ a general simplicial set.  To see this, recall the join $\star\colon \SSet \times \SSet \to \SSet$ of simplicial sets.  As explained in \cite{Joyal2002}, there is a natural inclusion $D \hookrightarrow A \star D$, whereby for fixed $D$, we obtain a functor $-\star D\colon \SSet \to D/\SSet$ with this value at $A$.  This has a right adjoint $D/\SSet \to \SSet$ which sends $t\colon D \to X$ to a simplicial set $X/t$. Now the limit of $t$ is defined to be a terminal object $\Delta^0 \to X/t$.  By adjointness, this amounts to a morphism $$\Delta^0 \star D \to X$$ extending $t$ along the inclusion $j\colon D \to \Delta^0 \star D$, which should be thought of as cones over $t$, satisfying lifting properties obtained, by adjointness, from those of \eqref{eq:terminal}.  

We work this time with the locally finitely presentable category
$(\Delta^0 \star D)| \SSet$, whose objects $(X,U)$ are simplicial
sets $X$ equipped with a set $U$ of marked cones $\Delta^0 \star D
\to X$, and whose morphisms are simplicial maps preserving marked
cones.
There is a full embedding $(\QCat_{D}) \hookrightarrow (\Delta^0
\star D)| \SSet$ sending a quasicategory with limits of shape $D$ to
its underlying simplicial set with all limit cones marked.
Now, given $(X,U) \in (\Delta^0 \star D)| \SSet$, the condition that a cone of $U$ is a limit cone amounts to asking that each 
\begin{equation}
\xymatrix{
(\partial \Delta^n \star D, \{n \star D\}) \ar[d]_{j_{n} \star D} \ar[rr]^{}  && (X,U)\\
( \Delta^n \star D,  \{n \star D\}) \ar@{.>}[urr] 
}
\end{equation}
has a filler. Similarly, injectivity with respect to the inclusion
\[j\colon (D, \varnothing) \to (\Delta^0 \star D,\{\id\}) \]
asserts that each diagram $t\colon D \to X$ admits a limit, with
limit cone in $U$.
The repleteness condition capturing the fact that each limit cone belongs to $U$ is expressed by injectivity with respect to
 \begin{equation*}
 \xymatrix{
 (J \star D,\{0 \star D\}) \ar[rr]^{\id} &&  (J \star D,\{0 \star D,1 \star D\})~,
 }
 \end{equation*}
 while finally we equip the inner horns
\[\{ \Lambda_{k}[n]\colon \Lambda^n_k \to \Delta^n\colon  0 < k < n\} \]
with the empty markings to encode the fact that the $X$ in $(X,U)$ is a quasicategory.

\subsubsection{Cartesian fibrations}
An inner fibration $p\colon X \to Y$ of simplicial sets is a morphism
having the right lifting property with respect to the inner horn
inclusions. Such an inner fibration is said to be a \emph{cartesian fibration} if each $1$-simplex $f\colon x \to py$ in $Y$ admits a lifting along $p$ to a \emph{p-cartesian} $1$-simplex $f^{\prime}\colon x^{\prime} \to y \in X$, where a morphism $g\colon a \to b \in X$ is said to be $p$-cartesian if each diagram
\begin{equation*}
\xymatrix{
\Delta^1 \ar[r] \ar[dr]_i \ar@/^1pc/[rr]^g & \Lambda^n_n \ar[d] \ar[r] & X \ar[d]^{p} \\
& \Delta^n \ar@{.>}[ur] \ar[r] & Y
}
\end{equation*}
where $i$ includes the vertices of $\Delta^1$ as the vertices $n-1,n\in\Delta^n$,
admits a filler as depicted.  A commutative square is said to be a morphism of cartesian fibrations if it preserves cartesian $1$-simplices.  As such, we obtain a category $\Cart(\SSet)$ of cartesian fibrations.

Let $\SSet^{\atwo}$ denote the arrow category.  Then an object of
$(\id\colon \Delta^1 \to \Delta^1)|\SSet^{\atwo}$ is a morphism
$p\colon X \to Y$ of simplicial sets together with a subset $U
\subseteq X_1$ of marked $1$-simplices, whilst a morphism is a
commutative square whose domain component preserves marked
$1$-simplices; we write $p\colon(X,U)\to Y$ for such an object.
We wish to describe $\Cart(\SSet)$ as a small injectivity class in $\id_{\Delta^1}|\SSet^{\atwo}$; that is, we shall describe a set of morphisms $\cj$ such that $p\colon (X,U) \to Y$ is $\cj$-injective if and only if $p$ is a cartesian fibration with $U$ the set of all $p$-cartesian $1$-simplices.  Note that $p$ is an inner fibration just when it is injective in $\SSet^{\atwo}$ with respect to each square 
\begin{equation*}
\xymatrix{
\Lambda^n_k \ar[d] \ar[r] & \Delta^n \ar[d]^{1} \\
\Delta^n \ar[r]_{1} & \Delta^n
}
\end{equation*}
with $0<k<n$. (Here we are employing the trick (see, for instance, Lemma~1 of
\cite{Bourke2019Equipping}) that $g$ has the right lifting property
with respect to $f$ if and only if $g$ is injective in the arrow
category with respect to the map $f \to 1_{\cod(f)}$ determined by $f$
and $1_{\cod(f)})$.
Thus equipping the above squares with the empty markings captures as injectives those $p\colon (X,U) \to Y$ with $p$ an inner fibration.  
To express the requirement that the elements of $U$ be cartesian $1$-simplices we use injectivity with respect to the squares
\begin{equation*}
\xymatrix{
(\Lambda^n_n,n- 1 \to n) \ar[d] \ar[rr] && (\Delta^n,n-1 \to n) \ar[d]^{1} \\
\Delta^n \ar[rr]_{1} && \Delta^n
}
\end{equation*}
whose sources have the unique $1$-simplex $n-1 \to n$ marked.
The existence of cartesian liftings is expressed by injectivity against
\begin{equation*}
\xymatrix{
(\Delta^0,\varnothing) \ar[d]^{d_0} \ar[rr]^{d_0} && (\Delta^1,0 \to 1) \ar[d]^{1} \\
\Delta^1 \ar[rr]_{1} && \Delta^1.
}
\end{equation*}
In order to ensure that each cartesian $1$-simplex belong to $U$ we add a repleteness condition based on the following lemma.
\begin{lemma}\label{lem:fibration}
Let $p\colon X \to Y$ be a cartesian fibration. Let $f\colon a \to b$
and $g\colon c \to b$ be 1-simplices in $X$, with $g$ a cartesian lifting of $(pf\colon pa \to pb,b)$ so that we obtain a $2$-simplex
\begin{equation*}
\xymatrix{
a \ar[dr]_{f} \ar[r]^{\alpha} & c \ar[d]^{g} \\
& b}
\end{equation*}
where $p\alpha\colon pa \to pa$ is degenerate.  Then $f\colon a \to b$ is cartesian if and only if $\alpha$ is an equivalence.
\end{lemma}
\begin{proof}
The degeneracy $p\alpha$ is an equivalence. Thus, by Lemma 2.4.1.5 of
\cite{Lurie}, $\alpha$ is an equivalence if and only it is cartesian.
And now since $g$ is cartesian, by Proposition 2.4.1.7 of
\cite{Lurie}, $\alpha$ is cartesian if and only if $f$ is.
\end{proof}

Given this, consider the pushout below left
\begin{equation*}
\xymatrix{
\Delta^1 \ar[d] \ar[r]^{d_2} & \Delta^2 \ar[d]^{} \\
J \ar[r]_{} & (\Delta^2)_{0 \cong 1}
}
\hspace{2cm}
\xymatrix{
0 \ar[dr] \ar[r]^{\simeq} & 1 \ar[d]^{} \\
& 2}
\end{equation*}
which is the generic $2$-simplex with $0 \to 1$ an equivalence.  There is a unique morphism $(\Delta^2)_{0 \cong 1} \to \Delta^1$ sending the equivalence $0 \to 1$ to the degeneracy on $0$, and $2$ to $1$.  Then injectivity of $p\colon (X,U) \to Y$ against the square below left 
\begin{equation*}
\xymatrix{
((\Delta^2)_{0 \cong 1},1 \to 2) \ar[d] \ar[rr] && ((\Delta^2)_{0 \cong 1},\{1 \to 2, 0 \to 2\}) \ar[d]^{1} \\
\Delta^1 \ar[rr]_{1} && \Delta^1
}
\hspace{1cm}
\xymatrix{
a \ar[dr]_{f} \ar[r]^{\alpha} & c \ar[d]^{g} \\
& b}
\end{equation*}
asserts: for any $2$-simplex in $X$ as on the right above in which $g \in U$ and $\alpha$ is an equivalence sent by $p$ to a degeneracy, the $1$-simplex $f$ is in $U$. Since each morphism of $U$ is cartesian this implies, by Lemma~\ref{lem:fibration}, that all cartesian morphisms belong to $U$.  In this way, we obtain the category of cartesian fibrations $\Cart(\SSet) \hookrightarrow (\id\colon \Delta^1 \to \Delta^1)|\SSet^{\atwo}$ as the full subcategory of injectives.  

Finally, to encode the full subcategory $\Cart(\QCat) \hookrightarrow \Cart(\SSet)$ of cartesian fibrations between \emph{quasicategories} we need to further encode that the target $Y$ of $p\colon (X,U) \to Y$ is a quasicategory --- of course it then follows that $X$ is also a quasicategory, since $p$ is an inner fibration. To this end, we add 
 the injectivity condition
\begin{equation*}
\xymatrix{
(\varnothing,\varnothing) \ar[d] \ar[r]^{1} & (\varnothing,\varnothing) \ar[d]^{p} \\
\Lambda^k_n \ar[r] & \Delta^n
}
\end{equation*}
for each inner horn.


\end{document}